\documentclass[reqno,11pt]{article}
\usepackage{amsmath, amsthm, amscd,amsfonts, amssymb,hyperref,floatrow}
\usepackage{graphicx, color,dsfont,xcolor}

\usepackage{authblk}

\def\Ha{{\mathcal H}}
\def\e{{\epsilon}}
\def\ra{{\rightarrow}}
\numberwithin{equation}{section}

\newtheorem{theorem}{Theorem}[section]
\newtheorem{lemma}[theorem]{Lemma}
\newtheorem{proposition}[theorem]{Proposition}

\newtheorem{rem}[theorem]{Remark}
\newtheorem{definition}[theorem]{Definition}

\def\tr{\mbox{tr}}

\renewcommand{\tilde}{\widetilde}          
\DeclareMathSymbol{\leqslant}{\mathalpha}{AMSa}{"36} 
\DeclareMathSymbol{\geqslant}{\mathalpha}{AMSa}{"3E} 
\DeclareMathSymbol{\eset}{\mathalpha}{AMSb}{"3F}     
\renewcommand{\leq}{\;\leqslant\;}                   
\renewcommand{\geq}{\;\geqslant\;}                   

\def \C{ \mathbb  C }
\newcommand{\R}{\mathbb{R}}

\newcommand{\N}{\mathbb{N}}

\def \P{ \mathbb P  }
\def \bP{ \mathbb P  }
\def \bE{ \mathbb E  }
\def \E{ \mathbb E  }
\newcommand \be  {\begin{equation*}}
\newcommand \bea {\begin{eqnarray} \nonumber }
\newcommand \ee  {\end{equation*}}

\newcommand \ba  {\begin{align}}
\newcommand \ea  {\end{align}}

\definecolor{remi}{rgb}{0,0,0}

\title{A diffusive matrix model for invariant $\beta$-ensembles}
\author[1]{Romain Allez}
\author[2]{Alice Guionnet}
\affil[1]{ Universit{\'e} Paris-Dauphine, Ceremade, 75\,016 Paris, France.}
\affil[2]{U.M.P.A.
ENS de Lyon 
46, all\'ee d'Italie, 
69364 Lyon Cedex 07 - France.}

\begin{document}
\maketitle
\begin{abstract}
We define a new diffusive matrix model converging towards the $\beta$-Dyson Brownian motion for all $\beta\in [0,2]$
that provides an explicit construction of  $\beta$-ensembles of random matrices that is invariant under the orthogonal/unitary group. 
We also describe the eigenvector dynamics of the limiting matrix process; we show  that when $\beta< 1$ and that two eigenvalues collide, the eigenvectors of these two colliding eigenvalues
 fluctuate very fast and take the uniform measure on the  orthocomplement of the eigenvectors of the remaining eigenvalues. 
\end{abstract}

\tableofcontents

\vspace{1cm}
\footnotesize

\normalsize
\section{Introduction}\label{intro}
It is well known that the law of the eigenvalues of the classical Gaussian matrix ensembles are given by a Gibbs measure of a Coulomb gas interaction with inverse temperature  $\beta=1$  (resp. 2, resp. 4) in the symmetric (resp. Hermitian, resp. symplectic) cases;
$$dP_\beta(\lambda)=\frac{1}{Z^\beta}\prod_{i<j}|\lambda_i-\lambda_j|^\beta e^{-\frac{1}{2}\sum \lambda_i^2}  \prod d\lambda_i\,.$$
 Such measures are associated with symmetric Langevin dynamics, the so-called Dyson Brownian motion, which describe  the random motion of the eigenvalues of a symmetric (resp. Hermitian, resp. symplectic) Brownian motion. They are given by the stochastic differential system
\begin{equation}
d\lambda_i(t)=db_i(t)-\lambda_i(t) dt +\beta\sum_{j\neq i}\frac{1}{\lambda_i(t)-\lambda_j(t)} dt
\end{equation}
with iid Brownian motions $(b_i)$.  These laws and dynamics have been intensively studied, and both local and global behaviours of these eigenvalues have been analyzed precisely, starting from the reference book of Mehta \cite{Mehta}.

More recently, the generalization of these distributions and dynamics to all $\beta\ge 0$, the so-called $\beta$-ensembles, was considered.  
 As for $\beta=1,2, 4$,  the Langevin  dynamics converge to their unique invariant Gibbs measure $P_\beta$
as times goes to infinity. Indeed, the  stochastic differential system under study is a set of Brownian motions in interaction according to a strictly convex potential. Thus, one can then show by a standard coupling argument that two solutions driven by the same Brownian motion but with different initial data will soon be very close to each others. This entails the uniqueness of the invariant measure as well as the convergence to this Gibbs measure.
It turns out   that the case $\beta\in [0,1)$ and the case $\beta\in [1,\infty)$ are quite different, as in the first case the eigenvalues process can cross whereas in the second the repulsion is strong enough so that the eigenvalues do not collide with probability one in finite time. However, the diffusion was shown to be well defined, even for $\beta<1$, by C\'epa and L\'epingle \cite{Cepa}, at list once reordered.

The goal of this article is to provide 
 a natural interpretation of $\beta$-ensembles in terms of random matrices
for $\beta\in [0,2]$.
Dumitriu and Edelman \cite{Dumitriu} already proposed a tridiagonal matrix 
with eigenvalues distributed according to the $\beta$-ensembles. However, this tridiagonal matrix lacks the invariant property of the classical ensembles. Our construction has this property and moreover is constructive as it is based on a dynamical scheme. It was proposed by JP Bouchaud, and this article provides rigorous proofs of the results stated in \cite{PRL}.  The idea is to interpolate between the Dyson Brownian motion and the standard Brownian motion by throwing a coin at every infinitesimal time step to decide whether our matrix will evolve according to a Hermitian Brownian motion (with probability $p$) or will keep the same eigenvectors but has eigenvalues diffusing according to a Brownian motion. When the size of the infinitesimal time steps goes to zero, we will prove that the dynamics of the eigenvalues of this matrix valued process converges towards the $\beta$-Dyson Brownian motion with $\beta=2p$. The same construction with a symmetric Brownian motion leads to the same limit with
 $\beta=p$. This result is more precisely stated in Theorem \ref{main}.  We shall not consider the extension to the symplectic Brownian motion in this paper, but  it is clear that the same result  holds
with  $\beta=4p$. Our construction can be extended to other matrix models such as Wishart matrices, Circular and Ginibre Gaussian Ensembles and will lead to similar results. 

We thus deduce from our construction that 
$\beta$-ensembles can be interpreted as an interpolation between free convolution (obtained by adding a Hermitian Brownian motion) and standard convolution (arising when the eigenvalues evolve following standard Brownian motions). It is natural to wonder whether a notion of $\beta$-convolution could be more generally defined.

Moreover we shall study the eigenvectors of our matrix-valued process. In the case where $\beta\ge 1$, their dynamics is well known and is similar to the dynamics of the eigenvectors of the Hermitian or Symplectic Brownian motions, see e.g. \cite{AGZ}.
When $\beta<1$ the question is to determine what happens at a collision.
It turns out that when we approach a collision, the eigenvectors of the non-colliding eigenvalues converge to some orthogonal family $B$ of $d-2$ vectors  whereas the eigenvectors of the colliding eigenvalues oscillate very fast and take the uniform distribution on the ortho-complement of $B$, see Proposition \ref{limit_phi_T_1}.

\section{Statement of the results}
Let $\Ha_d^\beta$ be the space of $d\times d$ symmetric (respectively
Hermitian) matrices if $\beta=1$ (resp. $\beta=2$) and $\mathcal{O}^\beta_d$ be the space of $d\times d$ orthogonal (respectively unitary) matrices if $\beta=1$ 
(resp.  $\beta=2$).

We consider the matrix-valued process defined as follows.
Let $\gamma$ be a positive real number 
and $M_0^\beta \in \Ha_d^\beta$ 
with distinct eigenvalues $\lambda_1 < \lambda_2 < \dots < \lambda_d$. 
For each $n\in \N$, we let $(\epsilon_k^{n})_{k \in \N}$ be a sequence of i.i.d $\{0,1\}$-valued Bernoulli variables with mean $p$ in the sense that 
\begin{equation*}
\P[\epsilon_k^{n} = 1] = p = 1-\P[\epsilon_k^{n} = 0]\,.
\end{equation*}
Furthermore, for $t\geq 0$, we set $\epsilon_t^{n} := \epsilon_{[nt]}^{n}$. 

In the following, the process $(H^\beta(t))_{t\geq 0}$ will denote a symmetric Brownian motion, i.e. a process with values in the set of $d\times d$ symmetric matrices
(respectively Hermitian if
$\beta=2$) with entries $H_{ij}^\beta(t), t\geq 0, i \leq j$ constructed via independent real valued Brownian 
motions $(B_{ij}, \tilde{B}_{ij}, 1\leq i \leq j \leq d)$ by 
\begin{align}\label{sym_brownian}
H_{ij}^\beta(t)= \left\{
    \begin{array}{ll}
    B_{ij}(t) + i (\beta-1) \tilde{B}_{ij}(t) & \mbox{if } i < j \\
    \sqrt{2}\, B_{ii}(t) & \mbox{otherwise}
    \end{array}
\right. 
\end{align}  
\begin{definition}\label{defM}For each $n\in \N$, we define a diffusive matrix process $(M_n^\beta(t))_{t\geq 0}$  such that $M_n^\beta(0):=M_0^\beta$ and for $t\geq 0$
\begin{equation}
dM_n^\beta(t) = -\gamma M_n^\beta(t) dt + \epsilon_t^n dH^\beta_t + (1-\epsilon_t^n) dY_t
\end{equation} 
where $(H_t^\beta)_{t\geq 0}$ is a $d \times d$ symmetric (resp. Hermitian) as defined in \eqref{sym_brownian}
whereas
\begin{align*}
dY_t = \sqrt{2} \sum_{i=1}^d \chi_i^n\left(\frac{[nt]}{n}\right) dB_t^i
\end{align*}
with i.i.d Brownian motions $(B_t^i)_{t\geq0}$ and 
where $\chi_i^n([nt]/n)$ is the spectral projector associated to the $i$-th eigenvalue $\lambda_i([nt]/n)$ of the matrix $M_n^\beta([nt]/n)$ if 
the eigenvalues are numbered as $\lambda_1([nt]/n) < \lambda_2([nt]/n) < \dots < \lambda_d([nt]/n)$ (we shall see that the above is possible as the eigenvalues are almost surely distinct at the given times $\{k/n, k\in\mathbb N\}$). 
\end{definition}
 
As for all $t$, the matrix $M_n^\beta(t)$ is in the space $\Ha_d^\beta$, we know that it can be decomposed as 
\begin{equation*}
M_n^\beta(t)=O_n^\beta(t) \Delta_n^\beta(t) O_n^\beta(t)^*
\end{equation*}  
where $\Delta_n^\beta(t)$ is the diagonal matrix whose diagonal is the vector of the ordered eigenvalues of $M_n^\beta(t)$
and where $O_n^\beta(t)$ is in the space $\mathcal{O}^\beta_d$ for all $t\in \R_+$.
We also introduce a matrix $O^\beta(0)$ to be the initial orthogonal matrix (resp. unitary if $\beta=2$) such that 
$M_0^\beta(t)=O^\beta(0) \Delta_0 {O^\beta(0)}^*$ where $\Delta_0:=\text{diag}(\lambda_1,\dots, \lambda_d)$.

The evolution of the eigenvalues
of $M_n^\beta(t)$ during the time interval $[k/n;(k+1)/n]$ is given by independent Brownian motions
if $\epsilon_k^n=0$ and by Dyson Brownian motions if $\epsilon_k^n=1$.

The eigenvectors of $M_n^\beta(t)$ do not evolve on intervals $[k/n;(k+1)/n]$ such that $\epsilon_k^n=0$ and evolve with the classical 
diffusion of the eigenvectors of Dyson Brownian motions if $\epsilon_k^n=1$ (see \cite{AGZ} for a review on Dyson Brownian motion).

Our main theorems describe the asymptotic properties of the ordered eigenvalues of the matrix $M_n^\beta(t)$ denoted in the following as
\begin{equation}\label{order_eigenvalues_def}
(\lambda^n_1(t)\leq \lambda_2^n(t) \leq \cdots\leq \lambda_d^n(t))
\end{equation} 
and also those of the matrix $O_n^\beta(t)$ defined above, as $n$ goes to infinity.

Let $(b_t^i)_{t\geq 0}, i \in\{1,\dots,d\}$ be a family of independent Brownian motions on $\R$.
Recall that C\'epa and L\'epingle showed in \cite{Cepa} the uniqueness
and existence of the strong solution to
the stochastic differential system
\begin{equation}\label{theeqlim}
d\lambda_i(t)=-\gamma \lambda_i(t) dt +\sqrt{2}db_t^i+\beta
p\sum_{j\neq i}\frac{1}{\lambda_i(t)-\lambda_j(t)} dt
\end{equation}
starting from $\lambda(0)=(\lambda_1\leq\lambda_2\leq\cdots\leq\lambda_d)$
and such that for all $t\geq 0$
\begin{equation}
\lambda_1(t)\le \lambda_2(t)\le\cdots \le \lambda_d(t)\quad a.s.
\end{equation}

For the scaling limit of the {\it ordered} eigenvalues, we shall prove that
\begin{theorem}\label{main}
Let  $M_0^\beta$ be  a symmetric (resp. Hermitian
)
matrix  if $\beta=1$ (resp. $\beta=2$)
with distinct eigenvalues $\lambda_1 < \lambda_2 < \dots < \lambda_d$ and $(M_n^\beta(t))_{t\ge 0}$ be the matrix process defined in Definition \ref{defM}.
Let $\lambda_1^n(t)\leq \dots\leq \lambda_d^n(t)$ be the ordered eigenvalues of the matrix $M_n^\beta(t)$. 
Let also $(\lambda_1(t),\dots,\lambda_d(t))_{t\geq 0}$ be the unique strong solution  of \eqref{theeqlim}
with initial conditions in $t=0$ given by $(\lambda_1, \lambda_2, \dots, \lambda_d)$.

Then, for any $T<\infty$,  the process $(\lambda_1^n(t),\dots,\lambda_d^n(t))_{t\in [0,T]}$ converges in law as $n$ goes to infinity towards 
the process $(\lambda_1(t),\dots,\lambda_d(t))_{t\in [0,T]}$ 
in the space of continuous functions $\mathcal{C}([0,T],\R^d)$ embedded with the uniform topology. 
\end{theorem}

In the case where $\beta p\geq 1$, the eigenvalues almost never collide
and we will see (see section \ref{pbetagrand}) in this case that it is easy to construct a coupling
of $\lambda$ and $\lambda^n$ so that $\lambda^n$ almost surely converges
towards $\lambda$.

We shall also describe the scaling limit of the matrix $O_n^\beta(t)$ (the columns of $O_n^\beta(t)$ are the eigenvectors of $M_n^\beta(t)$) 
when $n$ tends to infinity, at least until the first collision time for the eigenvalues, 
i.e. until the time $T_1$ defined 
as $T_1:=\inf\{t\geq 0: \exists i \in \{2,\dots,d\}, \lambda_i(t)=\lambda_{i-1}(t)\}$.

Let $w_{ij}^\beta(t), 1\leq i < j\leq d $ be a family of real or complex (whether $\beta=1$ or $2$) standard Brownian motions 
(i.e. 
$w_{ij}^\beta(t) = 
B_{ij}^1(t) + \sqrt{-1} \,(\beta-1) B_{ij}^2(t)$ where 
the $B_{ij}^1,B_{ij}^2$ are standard Brownian motions on $\R$), independent of the family of Brownian motions $(b_t^i)_{t\geq 0}, i \in \{1,\dots,d\}$. 
For $i<j$, set in addition $w_{ji}^\beta(t) := \bar{w}_{ij}^\beta(t)$ and define   
the skew Hermitian matrix (i.e. such that $R^\beta=-(R^\beta)^*$) by setting for $i\neq j$,
\begin{equation*}
dR_{ij}^\beta(t) =  \frac{dw_{ij}^\beta(t)}{\lambda_i(t)-\lambda_j(t)}, \quad R_{ij}^\beta(0)=0 \,.
\end{equation*}

Then, with $\lambda_{i}(t), 0\leq t \leq T_1, i \in \{1,\dots,d\}$ being the solution of \eqref{theeqlim} until its first collision time,  
there exists a unique strong 
solution $(O^\beta(t))_{0\leq t \leq T_1}$
to the stochastic differential equation
\begin{equation}\label{eq_ev_lim_mat}
dO^\beta(t) = \sqrt{p} O^\beta(t) d R^\beta(t) - \frac{p}{2} O^\beta(t) d\langle (R^\beta)^*,R^\beta\rangle_t
\end{equation}
This solution exists and is unique since it is a linear equation in $O^\beta$ and $R^\beta$ is a well defined martingale at least until time $T_1$.
It can be shown as in \cite[Lemma 4.3.4]{AGZ} that $O^\beta(t)$ is indeed an orthogonal (resp. unitary if $\beta=2$) matrix for all $t\in[0;T_1]$.

We mention at this point that the matrix $O_n^\beta(t)$ is not uniquely defined, {\it even} when we impose the diagonal matrix 
to have a non-decreasing diagonal $\lambda_1^n(t)\leq \dots \leq \lambda_n(t)$. Indeed, the matrix $O_n^\beta(t)$ can be replaced, 
for example, by $-O_n^\beta(t)$
(other possible matrices exist). The following proposition overcomes this difficulty. 

Define $T_n(1)$ to be the first collision time of the process $(\lambda_1^n(t),\dots,\lambda_d^n(t))$.
\begin{proposition}\label{good_O_n}
There exists a continuous process
$(O_n^\beta(t))_{0\leq  t\leq T_1}$ in $\mathcal{O}^\beta_d$ with a uniquely defined law and such that for each $t\in [0;T_n(1)]$, we have 
\begin{equation*}
O_n^\beta(t) \Delta_n^\beta(t) O_n^\beta(t)^*\stackrel{law}{=} M_n^\beta(t)\,,
\end{equation*}
where $\Delta_n^\beta(t)$ is the diagonal matrix of the ordered (as in \eqref{order_eigenvalues_def}) eigenvalues of $M_n^\beta(t)$. 
\end{proposition}
Proposition \ref{good_O_n} is proved in Section \ref{section_eigenvectors}. We are now ready to state our main result for the convergence 
in law of the matrix $O_n^\beta(t)$.

\begin{theorem}\label{theo_vectors} Let $\eta$ and $T$ be positive
real numbers. Then, conditionally on the sigma-algebra generated by
 $(\lambda_1^n(s),\dots,\lambda_d^n(s)), \\ 0\leq s\leq T_1\wedge T$, 
the matrix process $(O_n^\beta(t))_{0\leq t\leq (T_1- \eta)\wedge T}$ introduced in Proposition \ref{good_O_n} converges in law in the space of continuous functions 
$\mathcal{C}([0;T],\mathcal{O}^\beta_d)$
towards the unique solution of the stochastic differential equation \eqref{eq_ev_lim_mat}. 
%

\end{theorem}

Theorem \ref{theo_vectors} gives a convergence result as $n$ goes to infinity for the eigenvectors of the matrix process $(M_n^\beta(t))$ but only 
until the first collision time $T_1$. If $p\beta\geq 1$, the result is complete as one can show (see \cite{AGZ} and section \ref{pbetagrand}) that the process 
$(\lambda_1(t),\dots,\lambda_d(t))$ is a non colliding process (i.e. almost surely $T_1=\infty$). However, if $p\beta<1$, it would be interesting to 
have a convergence on all compact sets $[0;T]$ even after collisions occurred. Our next results describe the behavior of the columns of the matrix 
$O^\beta(t)$ denoted as
$(\phi_1(t),\dots,\phi_d(t))$ when $t\rightarrow T_1$ with $t<T_1$.  

We first need to describe the behavior of the eigenvalues $(\lambda_1(t),\dots,\lambda_d(t))$ in the left vicinity of $T_1$.

\begin{proposition}\label{first_coll_singularity}
If $p\beta <1$ then  almost surely $T_1<\infty$ and there exists a unique index $i^*\in \{2,\dots,d\}$ such that $\lambda_{i^*}(T_1)=\lambda_{i^*-1}(T_1)$. 
While we have, for all $t\geq 0$ and almost surely, 
\begin{equation*}
\int_0^{t} \frac{ds}{(\lambda_{i^*}-\lambda_{i^*-1})(s)} < +\infty\,, 
\end{equation*}
the following divergence occurs almost surely 
\begin{equation}\label{div_integral}
\int_0^{T_1} \frac{ds}{(\lambda_{i^*}-\lambda_{i^*-1})^2(s)} = +\infty\,. 
\end{equation}
\end{proposition}

The first part of Proposition \ref{first_coll_singularity} is proved in subsections \ref{regularity_prop_of_lim} and \ref{estimate_collisions}, the
last statement is proved in \ref{section_eigenvectors}.
Hence equality \eqref{div_integral} implies the existence of diverging integrals in the SDE \eqref{eq_ev_lim_mat}.
Because of this singularity, we will show 
\begin{proposition}\label{limit_phi_T_1}
Conditionally on $(\lambda_1(t),\dots,\lambda_d(t)),0\leq t\leq T_1$, we have: 
\begin{enumerate}
\item For all $j\neq i^*,i^*-1$, the eigenvector $\phi_j(t)$ for the eigenvalue $\lambda_j(t)$  converges almost surely to a vector denoted $\tilde\phi_j$
as $t$ grows to $T_1$. 
The family $\{\tilde\phi_j,j\neq i^*,i^*-1\}$ is an orthonormal family of $\R^d$ (respectively $\C^d$) if $\beta=1$ (resp. $\beta=2$). We denote 
by $V$ the corresponding generated subspace and 
by $W$ its two dimensional orthogonal complementary in $\R^d$ (resp. $\C^d$).
\item The family $\{\phi_{i^*}(t),\phi_{i^*-1}(t)\}$ converges weakly to the uniform law on the  orthonormal basis of $W$ as $t$ grows to $T_1$.
\end{enumerate}
\end{proposition}

The paper is organized as follows. In Section \ref{limit_eigenvalues_prop}, we review and establish some new properties for the limiting 
eigenvalues process $(\lambda_1(t),\dots,\lambda_d(t))$ defined in \ref{theeqlim} that will be useful later in our proof of 
Theorems \ref{main} and \ref{theo_vectors}.
We also introduce, in subsection \ref{approx_less_colliding}, a process with fewer collisions that approximates the limiting eigenvalue 
process. In fact this gives a new construction
of the limiting eigenvalues process already constructed in \cite{Cepa}, perhaps simpler and more intuitive using only standard It\^o's calculus. 
We give some useful estimates on the processes of eigenvalues and matrix entries of $M_n^\beta$ in Section \ref{eigenvalues_of_M_n}.
In Section \ref{conv_until_first_coll}, we prove the almost sure convergence of the process $(\lambda_1^n,\dots,\lambda_d^n)$ to the limiting 
eigenvalues process
$(\lambda_1,\dots,\lambda_d)$ until the first hitting time of two particles with a coupling argument. 
In Section \ref{end_proof_main}, we finish the proof of Theorem \ref{main} by approximating in the same way the process
$(\lambda_1^n,\dots,\lambda_d^n)$ with the same idea of separating the particles which collide by a distance $\delta >0$. At this point, it suffices  
to apply that the result of Section \ref{conv_until_first_coll} to show that the two approximating processes are close in the large $n$ limit. 
In Section \ref{section_eigenvectors}, we prove Theorem \ref{theo_vectors}, the last statement of Proposition \ref{first_coll_singularity} 
and Propositions \ref{good_O_n} and \ref{limit_phi_T_1}. 
 
\section{Properties of the limiting eigenvalues process}\label{limit_eigenvalues_prop}

In this section we shall study  the unique strong solution
of \eqref{theeqlim} introduced by C\'epa and L\'epingle in \cite{Cepa}. We first
derive some boundedness and smoothness properties. In view of proving the convergence of $\lambda^n$ towards this process, and in 
particular to deal with possible collisions, we construct it for $p\beta<1$ as the limit of a process which is defined similarly except when two
particles hit, when  we separate  them by a (small)  positive distance,
see Definition \ref{deflimprocdelta}.

\subsection{Regularity properties of the limiting process}\label{regularity_prop_of_lim}
\begin{lemma}\label{lemCepa}
Let $\lambda=(\lambda_1\leq \lambda_2\leq \cdots \leq \lambda_d)$.
Then there exists a unique strong solution of \eqref{theeqlim}.
Moreover, it satisfies 
\begin{itemize}
\item  For all $T<\infty$,    there exists $\alpha, M_0>0$ finite so that for $M\geq M_0$
\begin{equation}\label{zx}
\P\left[\max_{1\le i\le d}\sup_{0\leq t\leq T}|\lambda_i(t)|\geq M\right]\leq e^{-\alpha (M-M_0)^2} \,.
\end{equation}

\item For all $T<\infty$, all $i,j\in\{1,\ldots,d\}$, $i\neq j$,
$$\E\left[\int_0^T\frac{ds}{|\lambda_i(s)-\lambda_j(s)|}\right] <\infty\,.$$
Furthermore, there exists $\alpha, M_0>0$ finite so that for $M\geq M_0$ and 
$i\neq j$, we have $$\P\left[ \int_0^T\frac{ds}{|\lambda_i(s)-\lambda_j(s)|} \geq M \right] \leq e^{-\alpha (M-M_0)^2}\,.$$

\end{itemize}
\end{lemma}
{\bf Proof.}
The existence and unicity of the strong solution is \cite[Proposition 3.2]{Cepa}.

For the first point, we choose a twice continuously differentiable symmetric function $\phi$, increasing on $\mathbb R^+$, 
which approximates smoothly $|x|$ in the neighborhood
of the origin so
that $\phi(0)=0$, $x\phi'(x)\ge 0$, 
 $ |\phi'(x)|\le c$ and $|\phi''(x)|\le c$, whereas 
$|\phi(x)|\ge |x|\times |x|\wedge 1$ (take e.g $\phi(x)=x^2(1+x^2)^{-1/2}$)
to obtain by It\^o's Lemma
\begin{align*}
 d(\phi(\lambda_i(t))) &= - \gamma\lambda_i(t) \phi'(\lambda_i(t)) dt + \sqrt{2} \phi'(\lambda_i(t)) db^i_t \\ 
&+ p \beta \sum_{j\not= i}  \phi'(\lambda_i(t))\frac{dt}{\lambda_i(t)-\lambda_j(t)} +  \phi''(\lambda_i(t)) dt.
\end{align*}
For all $t$, we have $ \lambda_i(t) \phi'(\lambda_i(t)) \geq 0$, and also 
\begin{align*}
\sum_{i=1}^d  \sum_{j\not= i} \frac{ \phi'(\lambda_i(t))}{\lambda_i(t)-\lambda_j(t)}
&=\frac{1}{2} \sum_{i=1}^d  \sum_{j\not= i} \frac{\phi'(\lambda_i(t)) - \phi'(\lambda_j(t))}{\lambda_i(t)-\lambda_j(t)} \leq \frac{d(d-1)}{2} \mid\mid \phi''\mid\mid_\infty\,.
\end{align*}
We deduce from the above arguments that there exists $C>0$ such that 
\begin{align*}
\sum_{i=1}^d\phi(\lambda_i(t)) \leq \sqrt{2}\sum_{i=1}^d \int_0^t \phi'(\lambda_i(s)) db^i_s + C t+\sum_{i=1}^d\phi(\lambda_i) \,.
\end{align*}
By usual martingales inequality,
as $\phi'$ is uniformly bounded we know that, see e.g. \cite[Corollary H.13]{AGZ}, 
$$P\left[\sup_{0\le t\le T} \mid \sum_{i=1}^d\int_0^t\phi'(\lambda_i(t))db_i(t)\mid \ge  M \right] \leq \exp(-\frac{M^2}{2 c T })$$
and therefore using the fact that $|\phi(x)|\ge |x|\times |x|\wedge 1$, we deduce the first point with $M_0=| \sum_{i=1}^d\phi(\lambda_i)|+CT$ and $\alpha=1/2CT$.

For the second point, we first remark as in the proof of \cite[Lemma 3.5]{Cepa} that for all $i<d$ 
\begin{align*}
p \beta  \int_0^T \frac{dt}{\mid\lambda_d(t)-\lambda_i(t)\mid} &\leq  p \beta \sum_{j<d}  \int_0^T \frac{dt}{\mid\lambda_d(t)-\lambda_j(t)\mid} \\  
&= p \beta\sum_{j <d}  \int_0^T \frac{dt}{\lambda_d(t)-\lambda_j(t)} \\  
 &= \lambda_d(T)-\lambda_d(0) - \sqrt{2} b^d_T + \gamma\int_0^T \lambda_d(t)dt \,. 
\end{align*}
so that the first point gives the claim fo $j=d$.
We then continue recursively.
\qed


\subsection{Estimates on collisions}\label{estimate_collisions}
To obtain regularity estimates on the process $\lambda$,
we need to control the probability that more than two particles are close together.
We shall prove, building on an idea from C\'epa and L\'epingle \cite{Cepa2}, that
\begin{lemma}\label{multiple_collisions}
For $r\ge 3$ and $I\subset\{1,\ldots, d\}$ with $|I|=r$, set
$$S^I_t=\sum_{i,j\in I} (\lambda_i(t)-\lambda_j(t))^2\,.$$
We let, for $\varepsilon>0$,
$$\tau^r_\varepsilon:=\inf\{t\ge 0: \min_{|I|=r} S^I_t\le \varepsilon\}$$

Then, for any $T>0$ and $\eta>0$, for any $r\ge 3$
 there exists $\varepsilon_r>0$ which only depends
on $\{S^I_0, |I|\ge 3\}$ so that
$$\bP\left( \tau^r_{\varepsilon_r}\le T\right)\le \eta\,.$$
\end{lemma}
{\bf Proof.}
The proof is done by induction over $r$ and we start with the
case $r=d$, $I=\{1,\ldots,d\}$. Then, $S$ verifies the following SDE (see e.g. \cite[Theorem 1]{Cepa2}):
$$dS_t = -2\gamma S_t dt + 4 \sqrt{d} \sqrt{S_t} d\beta_t + a dt$$
where $\beta_t$ is a a standard brownian motion and $a=2d(d-1)(2+p\beta d)$.
The square root of $\rho_t:=\sqrt{S_t}$ verifies the SDE
$$d\rho_t = -\gamma \rho_t dt + 2\sqrt{d}\, d\beta_t + (\frac{a}{2}-2d) \frac{dt}{\rho_t}. $$
In particular, one can check that, if $\alpha=2-\frac{a}{4d} = 2 - (d-1)(1+p\beta d /2)$
$$d\rho_t^\alpha = - \alpha \gamma \rho_t^\alpha dt + 2\sqrt{d} \, \alpha \rho_t^{\alpha-1} d\beta_t. $$
Thus, as $\alpha<0$ for $d\ge 3$, for any $\varepsilon>0$, $\rho^{\alpha-1}_{t\wedge \tau_\varepsilon^d}$
is bounded so that $\int_0^. \rho^{\alpha-1}_{s\wedge \tau_\varepsilon^d} d\beta_s$ is a  martingale
and therefore
$$\E[\rho_{T\wedge \tau^d_\varepsilon}^\alpha] \le  \rho_0^\alpha - \alpha \gamma \int_0^T \E[\rho_{t\wedge \tau^d_\varepsilon}^\alpha] dt$$
By Gronwall's lemma, since $\sup_t \E[\rho_{t\wedge \tau^d_\varepsilon}^\alpha]$ is finite,  we deduce that
\begin{equation*}
\E[\rho_{T\wedge\tau_\varepsilon^d}^\alpha]  \le  \rho_0^\alpha (1-\frac{1}{\alpha \gamma}) e^{-\alpha \gamma T} + \frac{\rho_0^\alpha}{\alpha \gamma}.
\end{equation*}  
As a consequence, since  $\alpha <0$, we have
\begin{align*}
\varepsilon^{\alpha/2} \P(\tau^d_\varepsilon\le T) \leq \E[ S_{T\wedge \tau^d_\varepsilon}^{\alpha/2}] = 
\E[\rho_{T\wedge \tau^d_\varepsilon}^{\alpha}] 
\leq \rho_0^\alpha (1-\frac{1}{\alpha \gamma}) e^{-\alpha \gamma T} + \frac{\rho_0^\alpha}{\alpha \gamma}\,.\\ 
\end{align*} 
We can take $\varepsilon$ small enough to obtain the claim for $r=d$.

We next assume that we have proved the claim for $u\geq r+1$
and choose $\varepsilon_{r+1}$ so that the probability that the hitting
time is smaller than $T$ is smaller than $\eta/2$.
We can choose $I$ to be connected without loss of generality as the $\lambda^i$
are ordered.
We let $R=\min\{\tau^I_\varepsilon, \tau^{r+1}_{\varepsilon_{r+1}}\}$
when $\tau^I_\varepsilon$ is the first time where $S^I$ reaches $\varepsilon$.
Again following \cite{Cepa2}, we have
\begin{eqnarray}
\log S^I_{T\wedge R}&=&\log S^I_0 - 2 \gamma T +4\sqrt{2}\sum_{k,j\in I}\int_0^{T\wedge R}
\frac{\lambda_j(t)-\lambda_k(t)}{S^I_t} db^j_t \nonumber\\
&&+2\beta p\sum_{j,k\in I}\sum_{l\notin I}\int_0^{T\wedge R}\frac{\lambda_j(t)-\lambda_k(t)}{S^I_t}[\frac{1}{\lambda_j(t)-\lambda_l(t)}-\frac{1}{\lambda_k(t)-\lambda_l(t)}] dt\nonumber\\
&&+4r[(r-1)(\frac{p\beta}{2} r+1)-2]\int_0^{T\wedge R} \frac{dt}{S^I_t}\label{borning}
\end{eqnarray}
Note that $M_t=4\sqrt{2}\sum_{k,j\in I}\int_0^{t\wedge R}
\frac{\lambda_j(s)-\lambda_k(s)}{S^I_s} db^j_s $ is a martingale 
with bracket $A_t=16 r\int_0^{t\wedge R}\frac{ds}{S^I_s}$. 
For $r\ge 3$, $4r[(r-1)(r p\beta/2 +1)-2]\ge 2 p \beta>0$ and therefore we deduce
\begin{align*}
\E[ \log &S^I_{T\wedge R}] \geq \, \log S^I_0 -2\gamma T +2 \beta p\E\left[ \int_0^{T\wedge R} \frac{dt}{S^I_t}\right] \\
&+ \E\left[ 2\beta p\sum_{j,k\in I}\sum_{l\notin I}\int_0^{T\wedge R}\frac{\lambda_j(t)-\lambda_k(t)}{S^I_t}[\frac{1}{\lambda_j(t)-\lambda_l(t)}-\frac{1}{\lambda_k(t)-\lambda_l(t)}] dt\right]\\
\end{align*}
For $j,k\in I$, we cut the last integral over times 
$$\Omega_{j,k}=\{ t\le T\wedge R : \sum_{l\notin I}\frac{1}{\lambda_j(t)-\lambda_l(t)}\frac{1}{\lambda_k(t)-\lambda_l(t)}\le \frac{1}{S^I_t}\}$$
so that 
$$-\sum_{j,k\in I}
\int_{\Omega_{j,k}}\frac{(\lambda_j(t)-\lambda_k(t))^2}{S^I_t}\sum_{l\notin I}
[\frac{1}{(\lambda_j(t)-\lambda_l(t))(\lambda_k(t)-\lambda_l(t))}] dt\ge 
-\int_0^{T\wedge R} \frac{dt}{S^I_t}$$
This term will therefore be compensated by the third term in \eqref{borning}.
For the remaining term, if  $l\notin I$ is such that 
$\min_{i\in I} |\lambda_l-\lambda_i|\le \min_{i\in I} |\lambda_k-\lambda_i|$
for all $k\notin I$ then if $t\in \Omega_{j,k}^c$ and $i^*\in I$ is 
so that $\min_{i\in I} |\lambda_l-\lambda_i|=|\lambda_l-\lambda_{i^*}|$,
we get

$$ \frac{d-r}{(\lambda_l(t)-\lambda_{i^*}(t))^2}\ge \frac{1}{S^I_t}$$
and therefore on $\tau^{r+1}_{\varepsilon_{r+1}}\ge t$,
$$\varepsilon_{r+1}\le  S^I_{t}+\sum_{j\in I}
(\lambda_j(t)-\lambda_l(t))^2\le  S^I_t +2r (\lambda_{i*}(t)-\lambda_l(t))^2 
+2S^I_t \le (3+2r(d-r)) S^I_t\,.$$
As a consequence,
 we have the bound  for all  $j,k\in I$, all $t\in \Omega_{j,k}^c$, $t\le R$, 
$$\frac{\lambda_j(t)-\lambda_k(t)}{S_t^I} \geq -1/\sqrt{S^I_t}\geq -\sqrt{3+2r(d-r)}/\sqrt{\varepsilon_{r+1}}$$ 
which entails the existence of a finite constant $c$ so that
\begin{align*}
\sum_{j,k\in I}\sum_{l\notin I}
\int_{\Omega_{j,k}^c}\frac{\lambda_j(t)-\lambda_k(t)}{S^I_t}[&\frac{1}{\lambda_j(t)-\lambda_l(t)}-\frac{1}{\lambda_k(t)-\lambda_l(t)}] dt \\
&\geq - \frac{c}{\sqrt{\varepsilon_{r+1}}} \sum_{i\in I}\sum_{l\notin I}
\int_0^T \frac{dt}{\mid \lambda_i(t)-\lambda_l(t)\mid}\,.
\end{align*}
Using Lemma \ref{lemCepa} we hence conclude that there exists  a universal finite
constant $c'$ depending only on $T$ so that
\be
\E[\log S^I_{T\wedge R}] \geq \log S^I_0-2\gamma T  -\frac{c'}{\sqrt{\varepsilon_{r+1}}}\,.
\ee
On the other hand, we have 

$$\bE[\log S^I_{T\wedge R}]\leq \bP(\tau^I_\varepsilon \le T) \log(\varepsilon) +\bE[ \sup_{0\le t\le T} \log S^I_t]$$
where the last term is bounded above by \eqref{zx}.
We deduce that

$$\bP(\tau^I_\varepsilon\le T)\le \frac{|\log S^I_0|}{|\log(\varepsilon)|}+
\frac{ c''}{\sqrt{\varepsilon_{r+1}}|\log(\varepsilon)|} + \frac{c}{\mid\log(\varepsilon)\mid} + \frac{2\gamma T}{\mid \log(\varepsilon)\mid} \,.$$
We finally choose $\varepsilon$ small enough so that the right hand side is smaller than $\eta/2$ to conclude.
\qed

We next show that not only collisions of three particles are rare but also two collisions of different particles rarely happen around the same time.

\begin{lemma}\label{2_Couples}
For all $i,j$ such that $ i+1 < j$, set $$\tau_{\varepsilon'}^{ij} = \inf\{t\geq 0: (\lambda_i(t)-\lambda_{i-1}(t))^2 + (\lambda_j(t)-\lambda_{j-1}(t))^2 \leq \varepsilon' \} .$$
Then, for any $T>0$ and $\eta>0$, there exists $\varepsilon'$ 
such that $$\P\left[\tau_{\varepsilon'}^{ij} \leq T\right] \leq \eta.$$ 
\end{lemma}
{\bf Proof.}
Using It\^o's formula, it is easy to see that 
\begin{align*}
  d&\left(  (\lambda_i-\lambda_{i-1})^2 +  (\lambda_j - \lambda_{j-1})^2 \right) = 8(1+p\beta) dt \\ 
  &-2 \gamma \left[ (\lambda_i - \lambda_{i-1})^2 + (\lambda_j - \lambda_{j-1})^2 \right]dt      \\ 
  &+ 2\sqrt{2} \left[ (\lambda_i-\lambda_{i-1}) (db^i_t - db^{i-1}_t) + (\lambda_j - \lambda_{j-1}) (db^j_t - db^{j-1}_t )  \right] \\
  &-2p \beta \left[\sum_{k\not= i-1,i} \frac{(\lambda_i-\lambda_{i-1})^2}{(\lambda_i-\lambda_k)(\lambda_{i-1}-\lambda_k)}
  + \sum_{k\not= j-1,j} \frac{(\lambda_j-\lambda_{j-1})^2}{(\lambda_j-\lambda_k)(\lambda_{j-1}-\lambda_k)} \right]dt\,.
 \end{align*}
Set $X_t := (\lambda_i(t)-\lambda_{i-1}(t))^2 +  (\lambda_j(t) - \lambda_{j-1}(t))^2$ and note that the quadratic variation 
of $$ \int_0^t \frac{(\lambda_i-\lambda_{i-1}) (db^i_s - db^{i-1}_s) + (\lambda_j - \lambda_{j-1}) (db^j_s - db^{j-1}_s )}{\sqrt{X_s}} $$
is $2t$. Thus there exists a standard Brownian motion $B$ so that
\begin{align*}
dX_t &= 8 (1+p\beta) dt - 2\gamma X_t dt +4 \sqrt{X_t} dB_t \\  
&-2p \beta \left[\sum_{k\not= i-1,i} \frac{(\lambda_i-\lambda_{i-1})^2}{(\lambda_i-\lambda_k)(\lambda_{i-1}-\lambda_k)}
  + \sum_{k\not= j-1,j} \frac{(\lambda_j-\lambda_{j-1})^2}{(\lambda_j-\lambda_k)(\lambda_{j-1}-\lambda_k)} \right]dt\,.
\end{align*}
Note that, by the previous Lemma \ref{multiple_collisions},  we can choose $\varepsilon$ such that 
\begin{equation}\label{lemmeavant}
\P[\tau_\varepsilon^3 < T] \leq \frac{\eta}{2}.
\end{equation}
Moreover, for all $t\leq \tau_\varepsilon^3$ such that
 $X_t \leq \varepsilon/4$, we have for all $k \neq  i-1,i$,
$$(\lambda_i-\lambda_k)(\lambda_{i-1}-\lambda_k)(t) \geq  \frac{\varepsilon}{8}.$$
The same property holds for $j$.
To finish the proof, 
we will use the fact that the sum in the last term is bounded for all $t\leq \tau_\varepsilon^3$ such that $X_t \leq \varepsilon/4$. 
We thus need to introduce the process $Y_t$ 
defined by $Y_t = \min(X_t,\frac{\varepsilon}{4})$.
Let us set $f(x):=\min(x,\varepsilon/4)^{-p\beta}$. Note that $f$ is a convex function $\R_+ \rightarrow \R_+$ and that the left-hand derivative of $f$ is given by $$f'_-(x)=-p\beta x^{-p\beta-1} 1_{\{x\leq \frac{\varepsilon}{4}\}}.$$ 
Its second derivative in the sense of distributions is the positive measure 
$$f''(dx)=  p\beta \left(\frac{\varepsilon}{4}\right)^{-p\beta-1}  \delta_{\frac{\varepsilon}{4}} + \frac{p\beta(p\beta+1)}{x^{p\beta+2}} 1_{\{x\leq \frac{\varepsilon}{4}\}}\,dx\,.$$
Thus, by  It\^o-Tanaka formula, see e.g. \cite[Theorem 6.22]{karatzas},
 we have
\begin{align*}
Y_t^{-p\beta}& = Y_0^{-p\beta} - p \beta \int_0^t X_s^{-p\beta-1} 1_{\{X_s \leq \frac{\varepsilon}{4} \}} dX_s \\ 
&+ \frac{1}{2} 
\left(p \beta \left( \frac{\varepsilon}{4} \right)^{-p\beta-1} L_t^{\frac{\varepsilon}{4}}(X) + \int_0^{\frac{\varepsilon}{4}} \frac{p\beta(p\beta+1)}{x^{p\beta+2}} L_t^{x}(X) dx\right)\,,
\end{align*}
where $L_t^x(X)$ is the local time of $X$ in $x$.  
By definition we have  $$  \int_0^{\frac{\varepsilon}{4}} \frac{p\beta(p\beta+1)}{x^{p\beta+2}} L_t^{x}(X) dx = \int_0^t \frac{p\beta(p\beta+1)}{X_s^{p\beta+2}} 
1_{\{X_s\leq \frac{\varepsilon}{4}\}} d\langle X,X\rangle_s,$$ 
and thus, we obtain
\begin{align}\label{EDS_Y}
&Y_t^{-p\beta} =  Y_0^{-p\beta} + \int_0^t 1_{\{X_s \leq \frac{\varepsilon}{4} \}} \left(p\beta \gamma Y_s^{-p\beta} dt + 4 Y_s^{-p\beta-\frac{1}{2}} dB_s \right)\\
&+ 2 p^2\beta^2\int_0^t Y_s^{-p\beta-1}  \Bigg[\sum_{k\not= i-1,i} \frac{((\lambda_i-\lambda_{i-1})(s))^2}{((\lambda_i-\lambda_k)(s))((\lambda_{i-1}-\lambda_k)(s))}\notag \\
&+ \sum_{k\not= j-1,j} \frac{((\lambda_j-\lambda_{j-1})(s))^2}{((\lambda_j-\lambda_k)(s))((\lambda_{j-1}-\lambda_k)(s))} \Bigg]1_{X_s\le \varepsilon/4}
ds\notag + \frac{1}{2} 
p \beta \left( \frac{\varepsilon}{4} \right)^{-p\beta-1} L_t^{\frac{\varepsilon}{4}}(X) \notag\,.
\end{align} The definition of local time implies that, almost surely, $L_t^x(X)\leq t$.
We thus deduce from \eqref{EDS_Y} that 
\begin{align*}
\E\left[Y_{T\wedge \tau_{\varepsilon'}^{ij}\wedge \tau_{\varepsilon}^3}^{-p\beta} \right] \leq Y_0^{-p\beta} + \frac{1}{2} 
p \beta \left( \frac{\varepsilon}{4} \right)^{-p\beta-1} T +C
\int_0^T \E\left[Y_{t\wedge \tau_{\varepsilon'}^{ij}}^{-p\beta}\right]\,dt\,.
\end{align*}
with $C= (p\beta\gamma+4 p^2 \beta^2 (d-1)\frac{8}{\varepsilon})$.
Gronwall's Lemma implies that 
\begin{align}\label{E[Y^-pbeta]}
\E\left[Y_{T\wedge \tau_{\varepsilon'}^{ij}\wedge \tau_{\varepsilon}^3}^{-p\beta} \right] \leq \left(Y_0^{-p\beta} + \frac{1}{2} 
p \beta \left( \frac{\varepsilon}{4} \right)^{-p\beta-1} T\right)\exp(C T).
\end{align}
If $\varepsilon' < \varepsilon/4$, equation \eqref{E[Y^-pbeta]} implies that
\begin{equation}
(\varepsilon')^{-p\beta} \P\left[\tau_{\varepsilon'}^{ij} \leq T\wedge \tau_\varepsilon^3 \right] \leq Y_0^{-p\beta} \exp(C T),
\end{equation}
Taking $\varepsilon'$ small enough gives the result with \eqref{lemmeavant}. 
\qed

As a direct consequence, 
we deduce the  uniqueness  of the $i^*$ of Proposition \ref{first_coll_singularity}.
\begin{lemma}\label{unicity_i}
With the same notations as in the previous Lemma \ref{2_Couples}, we have almost surely
\begin{equation*}
\inf_{(k,\ell): k+1< \ell} \tau^{k\ell}_0 = + \infty.
\end{equation*}
In particular, this gives the unicity of the $i^*$ in Proposition \ref{first_coll_singularity}. 
\end{lemma}
\begin{proof} It is enough to write that for all $\varepsilon>0$
$$\P\left(\inf_{k+1<\ell} \tau_0^{k\ell}\le T\right)
\le d^2\{\max_{k+1<\ell}\P\left( \tau_0^{k\ell}\le T\wedge \tau^3_\varepsilon\right)
+\P\left(\tau^3_\varepsilon\le T\right)\}$$
and deduce from Lemmas \ref{2_Couples} and \ref{multiple_collisions} that
the right hand side is as small as wished when $\varepsilon$ goes to zero.

\end{proof}

\subsection{Smoothness properties of the limiting process}

\begin{lemma}\label{lambda_holder}
We have the following smoothness  properties:
\begin{itemize}
\item For all $T<\infty$ and $\varepsilon>0$, there exists
 $C,c',c$ finite positive constants so that for all $\delta,\eta$ positive real numbers so that
  $\eta\le c'(\varepsilon^2\wedge \delta\varepsilon)$ we have
\be
\P\left[\max_{1\le i\le d} \sup_{s\le t\le (s+ \eta) \wedge \tau_\varepsilon^3 \atop 0 \leq t\leq T}|\lambda_i(s)-\lambda_i(t)|\ge \delta
 \right]\le  \frac{C}{\eta}\left( 
e^{-c\delta^4/2\eta} +  e^{-c\varepsilon^4/\eta}\right) \,.
\ee

\item For all $T<\infty$ and $\varepsilon>0$, there exists
 $C,c',c$ finite positive constants so that for all $\delta,\eta$ positive real numbers so that
  $\eta\le c'(\varepsilon^2\wedge \delta\varepsilon)$ we have

\be
\P\left[ \max_{i\neq j} \sup_{s\le t\le (s+ \eta)\wedge \tau_\varepsilon^3 \atop 0\leq t \leq T} \int_s^t \frac{du}{|\lambda_i(u)-\lambda_{j}(u)|} \ge \delta  \right]\le  \frac{C}{\eta}\left( 
e^{-c\delta^4/2\eta} +  e^{-c\varepsilon^4/\eta}\right)
\,.
\ee
\end{itemize}
\end{lemma}
{\bf{Proof.}} Let us first fix $s\in [0,T]$ and set
$I=\{i\in\{2,\ldots, d\}: |\lambda_i(s)-\lambda_{i-1}(s)|\leq \varepsilon/3\}$
and note that on the event $\{s\leq \tau_\varepsilon^3 \}$,
 the connected subsets of $I$ contain at most one element.
Let $T_\varepsilon=\inf\{t\ge s: \inf_{ i\notin I}
|\lambda_i(t)-\lambda_{i-1}(t)|\leq\varepsilon/4\}$. The continuity of the $\lambda_i$ implies that $T_\varepsilon$ is almost surely strictly positive.

If $i\not\in I\cup\{I-1\}$, then we have, for $t\in [s;(s+\eta)\wedge \tau_\varepsilon^3 \wedge T_\varepsilon]$ 
\begin{align*}
|\lambda_i(t)-\lambda_i(s)| &\leq \gamma \int_s^t |\lambda_i(u)| du + \sqrt{2} |b^i_t - b^i_s| + p \beta \int_s^t \sum_{j \not= i } \frac{du}{|\lambda_i(u)-\lambda_{j}(u)|}  \\
 &\leq  \gamma \int_s^t |\lambda_i(u)| du + \sqrt{2} |b^i_t - b^i_s| + 4p \beta (d-1)\frac{ t-s }{\varepsilon} \,.   
\end{align*}
Using \eqref{zx} and \cite[Corollary H.13]{AGZ}, it is easy to deduce that there exists a constant $c>0$ such that  for $\eta<\varepsilon\delta/(8p\beta(d-1))$
\begin{equation}\label{i_not_in_I}
\P\left[\max_{i\not\in I\cup\{I-1\}} \sup_{t\in [s;(s+\eta)\wedge \tau_\varepsilon^3 \wedge T_\varepsilon]}  |\lambda_i(t)-\lambda_i(s)| \geq \delta\right] \leq c d e^{-\frac{\delta^2}{2\eta}}\,.
\end{equation} 
Now, if $i\in I$,  with the same argument as for \eqref{i_not_in_I} (the drift term in the SDE satisfied by  $\lambda_i +\lambda_{i-1}$ is also bounded), 
we can show that there exists a constant $c>0$ such that 
\begin{equation}\label{lambda_i+i-1}
\P\left[ \sup_{t\in [s;(s+\eta)\wedge \tau_\varepsilon^3 \wedge T_\varepsilon]}  
|(\lambda_i+ \lambda_{i-1})(t)-(\lambda_i+ \lambda_{i-1})(s) | \geq \delta\right] \leq c  e^{-c\frac{\delta^2}{2\eta}}\,.
\end{equation}
On the other hand, the process $x_i(t):=(\lambda_i-\lambda_{i-1})(t)$ verifies
\begin{align}
dx_i^2(t) &= 4(1+p\beta) dt - \gamma x_i^2(t) dt + 2 x_i(t) (db^i_ t - db^{i-1}_t)\nonumber\\ 
&- 2 p \beta \sum_{k\not= j-1,j} \frac{(\lambda_i(t) - \lambda_{i-1}(t))^2}{(\lambda_i(t) - \lambda_k(t)) (\lambda_{i-1}(t)-\lambda_k(t))} dt\,.\nonumber
\end{align}
The denominator in the last term of the above r.h.s
 is bounded below on the interval $t\in [s;(s+\eta)\wedge \tau_\varepsilon^3\wedge T_\varepsilon]$ by 
$2 p \beta (d-2)\frac{1}{\varepsilon}$. Thus, using again \eqref{zx} and \cite[Corollary H.13]{AGZ}, we can show that for $\delta>c\eta/\varepsilon$, 
\begin{equation}\label{lambda_i-i-1^2}
\P\left[ \sup_{t\in [s;(s+\eta)\wedge \tau_\varepsilon^3 \wedge T_\varepsilon]}  
|x_i(t) - x_i(s) | \geq \sqrt{\delta}\right]\le \P\left[ \sup_{t\in [s;(s+\eta)\wedge \tau_\varepsilon^3 \wedge T_\varepsilon]}  
|x_i^2(t) - x_i^2(s) | \geq \delta\right] \leq c e^{-c\frac{\delta^2}{2\eta}}\,
\end{equation}
where the first inequality is due to the fact that $x_i$ is non-negative.
Using \eqref{lambda_i+i-1} and  \eqref{lambda_i-i-1^2} gives for $\eta<\delta\varepsilon/c$
\begin{equation*}
\P\left[\max_{i\in I\cup\{I-1\}} \sup_{t\in [s;(s+\eta)\wedge \tau_\varepsilon^3 \wedge T_\varepsilon]}  |\lambda_i(t)-\lambda_i(s)| \geq \delta\right] \leq 2c d e^{-c\frac{\delta^4}{2\eta}}\,.
\end{equation*} 
Thus, with \eqref{i_not_in_I}, we deduce that for $\eta<\delta\varepsilon/c$
\begin{equation*}
\P\left[\max_{i} \sup_{t\in [s;(s+\eta)\wedge \tau_\varepsilon^3 \wedge T_\varepsilon]}  |\lambda_i(t)-\lambda_i(s)| \geq \delta\right] \leq 2c d e^{-c\frac{\delta^4}{2\eta}}\,.
\end{equation*}

In particular, there exists $c'>0$ so  that if $\varepsilon^2>c\eta$, 
$$\P\left[ T_\varepsilon< (s+\eta)\wedge\tau^3_\varepsilon\right]
\le \P\left[\max_{i} \sup_{s\le t\le (s+\eta)\wedge T_\varepsilon\wedge
\tau_\varepsilon^3}
|\lambda_i(t)- \lambda_i(s)|\geq 5\varepsilon/12 \right]\le \frac{4c d T}{\eta}
e^{-c'\varepsilon^4/2\eta}\,,$$
which is as small as wished provided $\eta$ is chosen small
enough.
This allows to remove the stopping time and get for any fixed $s<T$,
and $\delta>c\eta/\varepsilon$
$$\P\left[\max_{i}\sup_{s\le t\le (s+\eta)\wedge
\tau_\varepsilon^3}
|\lambda_i(t)- \lambda_i(s)|\geq \delta\right]\le 2 c d
e^{-c\delta^4/2\eta} +  2d
c e^{-c'\varepsilon^4/2\eta} \,.$$
 The uniform estimate
on $s$ is obtained as usual by taking $s$ in a grid with mesh $\eta/2$ up to
divise $\delta$ by two and to multiply the probability by $2T/\eta$. Thus we find  constant $c, c',$ and $ C$ so that if $\eta\le c(\varepsilon^2\wedge \delta\varepsilon)$ we have
$$\P\left[\max_{i}\sup_{s\le t\le (s+\eta)\wedge
\tau_\varepsilon^3\atop 0\le s,t\le T}
|\lambda_i(t)- \lambda_i(s)|\geq \delta\right]\le \frac{CT}{\eta}\left( 
e^{-c\delta^4/2\eta} +  e^{-c'\varepsilon^4/\eta}\right) \,.$$
The second control is a direct consequence of the
 first as 
we can first consider the cas $j=d$ to deduce that for $i<d$
$$|\int_s^t\frac{du}{\lambda_d(u)-\lambda_i(u)}|\le
|\lambda_d(t)-\lambda_d(s)|+\sqrt{2}|b_d(t)-b_d(s)|$$
where the right hand side is continuous.
We then consider recursively the other indices.
\qed

\subsection{Approximation by less colliding processes}\label{approx_less_colliding}
When $p\beta\geq 1$, it is well known \cite[Lemma 4.3.3]{AGZ} that the process $\lambda$
has almost surely no collision. In this case, the singularity of the drift
which defines the SDE is not really important as it is almost always avoided.
In the case $p\beta<1$, we know that collisions occur and in fact
can occur as much as for a Bessel process with small parameter.
The singularity
of the drift  becomes important, in particular when we will
 show the convergence in law of the process of the eigenvalues $\lambda^n$ towards $\lambda$. To this end, we show that $\lambda$ can be 
approximated by a process which does not spend too much time in collisions.

For $\delta >0$, we define a  new process $(\lambda_i^\delta(t))_{t\geq 0}$ as follows.
\begin{definition}\label{deflimprocdelta}
Let $T_1:= \inf\{t\geq 0 \,: \exists i \not = j, \lambda_i(t) = \lambda_j(t) \}$ and for all $t<T_1$, set $\lambda_i^\delta(t):=\lambda_i(t)$. 
For $t>T_1$, we define the process recursively by setting for all $\ell\geq 2, \lambda_i^\delta(T_\ell^\delta):=\lambda_i^\delta(T_\ell^\delta-)+i\delta$
and for $t>T_\ell^\delta$, the process $\lambda_i^\delta(t)$ is defined up to time 
$T_{\ell+1}^\delta:=\inf\{t > T_\ell^\delta \,: \exists i \not = j, \lambda_i^\delta(t) = \lambda_j^\delta(t) \}$ as the unique strong 
solution of the system
\begin{align}\label{sde_lamb}
d\lambda_i^\delta(t) = -\gamma \lambda_i^\delta(t) dt + \sqrt{2} db_t^i + p \beta \sum_{j\not = i} \frac{dt}{\lambda_i^\delta(t)-\lambda_j^\delta(t)}\,.
\end{align}
\end{definition}
 The main result of this section is that
\begin{theorem}\label{convproclim}
Construct the process $\lambda$ with the same Brownian motion $b$.
Then, for any time $T>0$, any $\xi\in (0,p\beta/4)$
$$\lim_{\delta\downarrow 0}\bP\left( \sup_{0\le t\le T}\max_{1\le i\le d}
|\lambda_i(t)-\lambda_i^\delta(t)|\le \delta^\xi\right)=1\,.$$
\end{theorem}
The theorem is a direct consequence of  the following lemma
and proposition. 
\begin{lemma}\label{lemjk} Let $\delta>0$.
Construct the process $\lambda$ with the same Brownian motion $b$ than $\lambda^\delta$.
There exists a constant $c>0$ such that, almost surely, for all $\ell\in \N$
\begin{align*}
\max_{1\leq i \leq d} \sup_{0 \leq t \leq T_{\ell}^\delta} | \lambda_i^{\delta}(t) - \lambda_i(t) | \leq c \delta \ell  \,.
\end{align*}
\end{lemma}
To finish the proof it is enough to show that $T_\ell^\delta$ goes
to infinity for $\ell \ll 1/\delta$. This is the content of the next proposition.

\begin{proposition}\label{T_Lgrand}
Let $T<\infty$, $0< \xi < p\beta/4$ and $L=[1/\delta^{1-\xi}]$. 
Then the probability $\P\left[T_L^\delta \leq T \right]$ vanishes when $\delta$ goes to zero. 
\end{proposition}

{\bf Proof of Lemma \ref{lemjk}.}
We proceed by induction over $\ell$ to show that, for each $\ell$,
$$ \sup_{0 \leq t \leq T_{\ell}^\delta} \left(\sum_{i=1}^d (\lambda_i^{\delta} - \lambda_i)^2(t)\right)^{1/2} \leq c\delta \ell $$ 
with $c = (\sum_{i=1}^d i^2 = d(d+1)(2d+1)/6)^{\frac 1 2}$.

$\bullet$ We treat the case $\ell=1$. By definition of the processes, $\lambda^\delta=\lambda$ on $[0,T_1^\delta)$. At time $t=T_1^\delta$, the separation procedure implies that
$$
\sum_{i=1}^d (\lambda_i^{\delta} - \lambda_i)^2(T_1^\delta) = \sum_{i=1}^d ((\lambda_i^{\delta} - \lambda_i) (T_1^\delta-) + i \delta)^2 =c^2\delta^2 \,.$$
The property is true for $\ell=1$.

$\bullet$ Suppose it is true for $\ell$.
For $t\in [T_\ell^\delta, T_{\ell+1}^\delta)$, since $\lambda^\delta$ and $\lambda$ are driven by the same Brownian motion, we get
\begin{align*}
d\sum_{i=1}^d &(\lambda_i^{\delta}(t) - \lambda_i(t))^2
=-2 \gamma \sum_{i=1}^d (\lambda_i^{\delta}(t) - \lambda_i(t))^2  dt \\
&+ 2p\beta\sum_{i=1}^d  \sum_{j\neq i}(\lambda_i^{\delta}(t) - \lambda_i(t))
\left(\frac{1}{\lambda_i^{\delta}(t)-\lambda_j^{\delta}(t)}-\frac
{1}{\lambda_i(t)-\lambda_j(t)} \right)dt\,.
\end{align*}
Observe that 
\begin{align}\label{astuce}
&\sum_{i=1}^d  \sum_{j\neq i}(\lambda_i^{\delta}(t) - \lambda_i(t))
\left(\frac{1}{\lambda_i^{\delta}(t)-\lambda_j^{\delta}(t)}-\frac
{1}{\lambda_i(t)-\lambda_j(t)} \right)
\\ \notag
&=\frac{1}{2}\sum_{i=1}^d  \sum_{j\neq i}(\lambda_i^{\delta}(t)-\lambda_j^{\delta}(t) -( \lambda_i(t)-\lambda_j(t)))
\left(\frac{1}{\lambda_i^{\delta}(t)-\lambda_j^{\delta}(t)}-\frac
{1}{\lambda_i(t)-\lambda_j(t)} \right)
\\ \notag
&=\frac{1}{2}\sum_{i=1}^d  \sum_{j\neq i}\left(\lambda_i^{\delta}(t)-\lambda_j^{\delta}(t) -( \lambda_i(t)-\lambda_j(t))\right)^2 
\frac{1}{(\lambda_i^{\delta}(t)-\lambda_j^{\delta}(t))(
\lambda_i(t)-\lambda_j(t))} 
\\ \notag
&\leq 0 
\end{align}
as the $(\lambda_i)_{1\le i\le d}$ and the $(\lambda_i^\delta)_{1\le i\le d}$ are ordered. Thus,
\begin{equation}\label{intervalle}
\sup_{t\in [T_\ell^\delta, T_{\ell+1}^\delta)}
\sum_{i=1}^d (\lambda_i^{\delta}(t) - \lambda_i(t))^2 \leq 
\sum_{i=1}^d (\lambda_i^{\delta}(T_\ell^\delta) - \lambda_i(T_{\ell}^\delta))^2.
\end{equation}
In addition, because of the separation procedure at time $T_{\ell+1}^\delta$, we have
\begin{align*}
&\left(\sum_{i=1}^d (\lambda_i^{\delta} - \lambda_i)^2(T_{\ell+1}^\delta)\right)^{1/2} = \left(\sum_{i=1}^d \left((\lambda_i^{\delta} -\lambda_i)(T_{\ell+1}^\delta-) + i\delta\right)^2\right)^{1/2} \\ 
&\leq \left(\sum_{i=1}^d (\lambda_i^{\delta} -\lambda_i)^2 (T_{\ell+1}^\delta-)  \right)^{1/2} + \delta  c\leq \delta (\ell+1)c \,, 
\end{align*}
where we used the induction hypothesis in the last line. The proof is thus complete.

\qed

{\bf Proof of Proposition \ref{T_Lgrand}.} In the case $p\beta\ge 1$, it
is well known  \cite[p. 252]{AGZ} that $T_1$ is almost surely infinite and therefore the
proposition is trivial. We hence restrict ourselves to $p\beta\le 1$.
Let $\eta >0$.
Let us define the stopping times 
\begin{align*}
\tau_\varepsilon^{3,\delta} &:= \inf\{t\geq 0: \min_{|I|=3} S_t^{I,\delta} \leq \varepsilon \}\,,\\
\tau_\varepsilon^{2,\delta}&:= \inf\{t\geq 0: \min_{1\leq i,j \leq d} ((\lambda_i^\delta-\lambda_{i-1}^\delta)^2+(\lambda_j^\delta-\lambda_{j-1}^\delta)^2)(t) \leq \varepsilon \},
\end{align*}
where $S_t^{I,\delta} := \sum_{i,j\in I} (\lambda_i^\delta-\lambda_j^\delta)^2(t)$. 
Set also $\tau_\varepsilon^\delta := \tau_\varepsilon^{2,\delta} \wedge \tau_\varepsilon^{3,\delta}.$
We know from Lemmas \ref{multiple_collisions} and \ref{2_Couples} that we can choose $\varepsilon$ small enough so that
\begin{align*}
\P\left[ \tau_{2\varepsilon}^3 \wedge \tau_{2\varepsilon}^2 \leq T \right] \leq \eta.
\end{align*}
The number $\varepsilon$ being fixed, it is then straightforward to see from 
 Lemma \ref{lemjk} that there exists $\delta_0$ small enough so that 
for all $\delta \leq \delta_0$, we have 
\begin{align*}
\P\left[\tau_{\varepsilon}^{\delta} \leq T  \right] \leq \eta.
\end{align*}
Now, we have 
\begin{align*}
\P\left[ T_L^\delta \leq T \right] \leq \eta + 
\P\left[\delta^\xi \sum_{\ell=1}^L 1_{\{T_{\ell+1}^\delta - T_\ell^\delta \geq \delta^\xi \}} \leq T; \tau_\varepsilon^\delta\geq T_L^\delta \right]\,.
\end{align*}
We  need to show that the second term goes to $0$ when $\delta\rightarrow 0$. Let $\{\mathcal F_t\}_{t\ge 0}$ be the filtration of the driving Brownian motion. We will prove in  Lemma \ref{lowerbound}, there exists a constant $c>0$ such that, on the event $\{\tau_\varepsilon^\delta\geq T_L^\delta\}$,   almost surely 
$$\sum_{\ell=1}^L \P\left[T_{\ell+1}^\delta-T_\ell^\delta \geq \delta^\xi \mid \mathcal{F}_{T_\ell^\delta}\right] \geq  
c\, \delta^{-p\beta+\xi}\,.$$
In the following, we suppose that $\delta$ is small enough so that $c\,\delta^{-p\beta+\xi} \geq \delta^{-p\beta+2\xi}$
and $\delta^{-\xi}\, T - \delta^{-p\beta+\xi} \leq -\delta^{-p\beta+2\xi}.$
For such $\delta$, we have
\begin{align*}
 \P&\left[\sum_{\ell=1}^L 1_{\{T_{\ell+1}^\delta - T_\ell^\delta \geq \delta^\xi \}} \leq \delta^{-\xi}\,T; \tau_\varepsilon^\delta\geq T_L^\delta \right] \\
& \leq \P\left[\sum_{\ell=1}^L 1_{\{T_{\ell+1}^\delta - T_\ell^\delta \geq \delta^\xi \}} - 
 \P\left[T_{\ell+1}^\delta-T_\ell^\delta \geq \delta^\xi \mid \mathcal{F}_{T_\ell^\delta}\right]  \leq - \delta^{-p\beta+2\xi}; 
 \tau_\varepsilon^\delta\geq T_L^\delta \right]\\
 &\leq \P\left[\left| \sum_{\ell=1}^L 1_{\{T_{\ell+1}^\delta - T_\ell^\delta \geq \delta^\xi \}} - 
 \P\left[T_{\ell+1}^\delta-T_\ell^\delta \geq \delta^\xi \mid \mathcal{F}_{T_\ell^\delta}\right]  \right| \geq  \delta^{-p\beta+2\xi}; 
 \tau_\varepsilon^\delta\geq T_L^\delta\right]\\
 &\leq \,\delta^{2p\beta-4\xi}\, \sum_{\ell=1}^L \P\left[T_{\ell+1}^\delta-T_\ell^\delta \geq \delta^\xi ; 
 \tau_\varepsilon^\delta\geq T_L^\delta\right]
\end{align*}
where we used the Tchebychev inequality in the last line. 
Using Lemma \ref{upperbound}, we get that there exists a constant $C>0$ such that 
\begin{align*}
 \P\left[\sum_{\ell=1}^L 1_{\{T_{\ell+1}^\delta - T_\ell^\delta \geq \delta^\xi \}} \leq \delta^{-\xi}\,T; \tau_\varepsilon^\delta\geq T_L^\delta \right] 
& \leq  C \, \delta^{2p\beta-4\xi}\, L\, \delta^{(1-p\beta)(1-2^{-1}\xi)} \leq C \, \delta^{p\beta-4\xi}\,
\end{align*}
which goes to $0$ when $\delta$ goes to $0$. The proposition is proved. 
\qed

\begin{lemma}\label{upperbound}
Let $\xi \in (0;2)$. 
Then there exists a constant $C>0$ such that, almost surely, on ; 
$ \tau_\varepsilon^\delta\geq T_L^\delta$
\begin{equation}
 \P\left[\delta^\xi \leq T_{\ell+1}^\delta-T_\ell^\delta \mid \mathcal{F}_{T_\ell^\delta} \right] \leq C \delta^{(1-p\beta)(1-2^{-1}\xi)}\,.
\end{equation}
\end{lemma}
{\bf Proof.}
We know that there are no multiple collisions nor simultaneous collisions (because of Lemmas \ref{multiple_collisions} 
and \ref{2_Couples})
and therefore we can denote by $i$ the unique element such that $\lambda_i^\delta(T_\ell^\delta-)=\lambda_{i-1}^\delta(T_\ell^\delta-)$ 
and $(\lambda_i^\delta-\lambda_{i-1}^\delta)(T_\ell^\delta) = \delta$.
We have by It\^o's formula
\begin{align}
&d(\lambda_i^{\delta}-\lambda_{i-1}^{\delta})(t) 
= - \gamma  (\lambda_i^{\delta}-\lambda_{i-1}^{\delta})(t) dt   \label{edsalpha1} + \sqrt{2} (db^i_t-db^{i-1}_t) \\ 
&+2p\beta \frac{dt}{(\lambda_i^{\delta}-\lambda_{i-1}^{\delta})(t)}
-\beta p \sum_{k\neq i,i-1}\frac{(\lambda_i^{\delta}-\lambda_{i-1}^{\delta})(t)}{(\lambda_i^{\delta}-\lambda_k^{\delta})(t)(\lambda_{i-1}^{\delta}-\lambda_k^{\delta})(t)}dt\,.\notag
\end{align}
Let us define the Bessel like process $(X_t)_{t\geq 0}$ by $X_{0}=\delta$ and for $t\geq 0$, 
\begin{equation}
dX_t = \sqrt{2}(db_{T_\ell^\delta+t}^i - db_{T_\ell^\delta+t}^{i-1}) + 2 p \beta \frac{dt}{X_t}.
\end{equation}
Using the comparison theorem for SDE \cite[Proposition 2.18]{karatzas} (note that the drifts are smooth before $T_{\ell+1}^{\delta}-T_\ell^\delta$), we know that for all $t\in [0, T_{\ell +1}^\delta-T_\ell^\delta)$, we have almost surely  
 \begin{equation}
 (\lambda_i^\delta-\lambda_{i-1}^\delta)(T_\ell^\delta+t) \leq X_t.
 \end{equation}
Let us define $T_X^\delta:=\inf\{t\geq 0: X_t = 0\}$. It is clear that almost surely $T_{\ell+1}^\delta -T_\ell^\delta \leq T_X^\delta.$
We thus have on $
 \tau_\varepsilon^\delta\geq T_L^\delta$
\begin{equation*}
\P\left[\delta^\xi \leq T_{\ell+1}^\delta-T_\ell^\delta \mid \mathcal{F}_{T_\ell^\delta} \right] \leq 
\P\left[ T_X^\delta \geq \delta^\xi \right].
\end{equation*}
We finally conclude using a classical result for Bessel process, see e.g. \cite[(13)]{yor}; the density with respect to the Lebesgue measure on $\R_+$ of the law of the random variable $T_X^\delta$ is 
\begin{equation*}
p_\delta(t) = \frac{1}{\Gamma(\frac{1-p\beta}{2})}
 \frac{1}{t} \left(\frac{\delta^2}{2 t}\right)^{\frac{1-p\beta}{2}} e^{-\frac{\delta^2}{2t}}\,.
\end{equation*}
Hence we deduce that for $\xi\le 2$ there exists a constant $c>0$ such that   
$$\P\left[ T_X^\delta \geq \delta^\xi \right] \leq c \,\delta^{(1-p\beta)(1-2^{-1}\xi)} .$$ 
\qed

For time $t\in [0;T]$, we define the random set 
\begin{equation}\label{def_I}
I_t:= \{i \in \{2,\dots,d\} : |\lambda_i^\delta - \lambda_{i-1}^\delta|(t) \leq \sqrt{\varepsilon}/3  \}.
\end{equation}
Note that, on the event $\Omega := \{ \tau_{\varepsilon}^{\delta} \geq T \}$, 
for each $t\leq T$, the set $I_t$ contains at most one element. 
For each $\ell \in \{1,\dots,L\}$, and $i\in \{1,\ldots, d\}$,
 we define the stopping times 
\begin{align*}
t_\ell^\delta(\sqrt{\varepsilon}/3):= \inf\{t\geq T_\ell^\delta : \min_{ j } |\lambda_j^\delta - \lambda_{j-1}^\delta|(t) \geq\sqrt{\varepsilon}/3 \}\,,\\
\bar{t}_\ell^\delta(i,\sqrt{\varepsilon}/6):= \inf\{t\geq T_\ell^\delta : \min_{ j\neq i } |\lambda_j^\delta - \lambda_{j-1}^\delta|(t) \leq \sqrt{\varepsilon}/6 \}\,.
\end{align*}
If $i$ denotes the unique index such that $\lambda_i^\delta(T_\ell^\delta-)=\lambda_{i-1}(T_\ell^\delta-)$,
note that if $T_\ell^\delta \leq \tau_\varepsilon^\delta$ then $\min_{ j\neq i } |\lambda_j^\delta - \lambda_{j-1}^\delta|(T_\ell^\delta) \geq \sqrt{\varepsilon}/3.$
\begin{lemma} \label{tps} 
If 
$T_\ell^\delta \leq \tau_\varepsilon^\delta$ and if $i$ denotes the (unique) index such that $\lambda_i^\delta(T_\ell^\delta-) = \lambda_{i-1}^\delta(T_\ell^\delta-)$ ,
then there exists a constant $c>0$ and $\delta_0>0$ such that for all $\delta \leq \delta_0$, we have
\begin{equation}
c \delta^{1-p\beta}\leq \P\left[t_\ell^\delta(\sqrt{\varepsilon}/3) \wedge  \bar{t}_\ell^\delta(i,\sqrt{\varepsilon}/6) \leq T_{\ell+1}^\delta  | \mathcal{F}_{T_\ell^\delta}\right].
\end{equation}
\end{lemma}

{\bf Proof.}
Note that $i$ is the unique element of the set $I_{T_\ell^\delta}$ defined by \eqref{def_I} for which $|\lambda_i^\delta-\lambda_{i-1}^\delta|(T_\ell^\delta) = \delta$.
For $\alpha = 1 - p \beta$ and $t\in [T_\ell^\delta;T_{\ell+1}^\delta)$, we have by It\^o's formula
\begin{align}
&d(\lambda_i^{\delta}-\lambda_{i-1}^{\delta})^\alpha(t) 
= - \gamma \alpha (\lambda_i^{\delta}-\lambda_{i-1}^{\delta})^\alpha(t) dt   \label{edsalpha2}  \\ 
&+\alpha (\lambda_i^{\delta}-\lambda_{i-1}^{\delta})^{\alpha -1}(t) \, \sqrt{2} (db^i_t-db^{i-1}_t) 
-\beta p \sum_{k\neq i,i-1}\frac{(\lambda_i^{\delta}-\lambda_{i-1}^{\delta})^\alpha(t)}{(\lambda_i^{\delta}-\lambda_k^{\delta})(t)(\lambda_{i-1}^{\delta}-\lambda_k^{\delta})(t)}dt\,.\notag
\end{align}
For $t\in [T_\ell^\delta, \tau_\varepsilon^\delta]$, we deduce that

$$d (\lambda_i^{\delta}-\lambda_{i-1}^{\delta})^\alpha(t) 
\ge \alpha
 (\lambda_i^{\delta}-\lambda_{i-1}^{\delta})^{\alpha -1}(t)
 \, \sqrt{2} (db^i_t-db^{i-1}_t) -c' (\lambda_i^{\delta}-\lambda_{i-1}^{\delta})^\alpha(t)dt $$
where $c' = \alpha \gamma+\beta p (d-2) 36/\varepsilon$.
Let $T^{\delta,\kappa}_{\ell+1}$ be the first time after $T_\ell^\delta$
so that $\lambda_i^\delta-\lambda_i^{\delta-1}$ reaches $\kappa<\delta$.
Then, as $\int_0^{.\wedge T^{\delta,\kappa}_{\ell+1}}   (\lambda_i^{\delta}-\lambda_{i-1}^{\delta})^{\alpha -1}(t)
 \, \sqrt{2} (db^i_t-db^{i-1}_t)$ is  a martingale,  we find  that
\begin{align}\label{lambda_i - lambda_{i-1}}
\E\left[ (\lambda_i^{\delta}-\lambda_{i-1}^{\delta})^\alpha(t_\ell^\delta(\sqrt{\varepsilon}/3) \wedge  \bar{t}_\ell^\delta(i,\sqrt{\varepsilon}/6)
 \wedge T_{\ell+1}^{\delta,\kappa}) \, \mid 
\mathcal{F}_{T_\ell^{\delta,\kappa}}\right]  
\geq 
\delta^\alpha \exp\left(-c' \,T \right).
\end{align}
Before time $ \bar{t}_\ell^\delta(i,\sqrt{\varepsilon}/6)$, $(\lambda_j^\delta - \lambda_{j-1}^\delta)(t)$ can not cancel if $j\neq i$. Therefore we can choose $\kappa$ small enough so that the last inequality implies

\begin{align*}
\E\left[ (\lambda_i^{\delta}-\lambda_{i-1}^{\delta})^\alpha(t_\ell^\delta(\sqrt{\varepsilon}/3) \wedge \bar{t}_\ell^\delta(i,\sqrt{\varepsilon}/6) ) \, 
1_{\{t_\ell^\delta(\sqrt{\varepsilon}/3) \wedge \bar{t}_\ell^\delta(i,\sqrt{\varepsilon}/6) \leq T_{\ell+1}^\delta \}}
\mid \mathcal{F}_{T_\ell^\delta}\right]  
\geq \frac 1 2
\delta^\alpha \exp\left(-c' \,T \right).
\end{align*}
which can be  rewriten using the fact that $|\lambda_i^\delta - \lambda_{i-1}^\delta |(t_\ell^\delta(\sqrt{\varepsilon}/3) \wedge \bar{t}_\ell^\delta(i,\sqrt{\varepsilon}/6)) \leq \sqrt{\varepsilon}/3$, as follows
\begin{equation*}
\P\left[t_\ell^\delta(\sqrt{\varepsilon}/3) \wedge \bar{t}_\ell^\delta(i,\sqrt{\varepsilon}/6) \leq T_{\ell+1}^\delta \mid \mathcal{F}_{T_\ell^\delta}\right] \geq 
\delta^\alpha \, \left(\frac{3}{\sqrt{\varepsilon}}\right)^{\alpha}    
\exp(-c'\,T) \,.
\end{equation*}
The lemma follows with $c= (\frac{3}{\sqrt{\varepsilon}})^{\alpha} \exp(-c'\,T)$.
\qed

\begin{lemma}\label{lowerbound}
Let $\xi, T >0$. There exists  a constant $c>0$  and $\delta_0>0$ 
so that
if $\delta\le \delta_0$, on  $T_\ell^\delta\leq \tau_\varepsilon^\delta\wedge T$,
\begin{equation}
 \P\left[\delta^\xi \leq T_{\ell+1}^\delta-T_\ell^\delta
\mid \mathcal{F}_{T_\ell^\delta} \right] \geq c \delta^{1-p\beta}\,.
\end{equation}
\end{lemma}
{\bf Proof.}We assume in the sequel that $\delta\le 1$. The proof is based on 
 Lemma \ref{tps}.  It implies 
\begin{align*}
\P&\left[\delta^\xi \leq T_{\ell+1}^\delta-T_\ell^\delta \mid \mathcal{F}_{T_\ell^\delta} \right] \\ 
&\geq 
\P\left[ 
t_\ell^\delta(\sqrt{\varepsilon}/3) \wedge  \bar{t}_\ell^\delta(i,\sqrt{\varepsilon}/6) \leq T_{\ell+1}^\delta ;
\delta^\xi \leq T_{\ell+1}^\delta- T_\ell^\delta\le 1 \mid \mathcal{F}_{T_\ell^\delta} \right]\,.  
\end{align*} 
By Lemma \ref{tps}, we deduce that
\begin{align*}
\P&\left[\delta^\xi \leq T_{\ell+1}^\delta-T_\ell^\delta \le  1 \mid \mathcal{F}_{T_\ell^\delta} \right] \\ 
&\geq c \delta^{1-p\beta} - 
\P\left[ t_\ell^\delta(\sqrt{\varepsilon}/3) \wedge  \bar{t}_\ell^\delta(i,\sqrt{\varepsilon}/6) \leq T_{\ell+1}^\delta\leq T+1 ; 
\delta^\xi \geq T_{\ell+1}^\delta- T_\ell^\delta \mid \mathcal{F}_{T_\ell^\delta} \right]\,.
\end{align*} 
But 
\begin{align*}
\P&\left[ t_\ell^\delta(\sqrt{\varepsilon}/3) \wedge  \bar{t}_\ell^\delta(i,\sqrt{\varepsilon}/6) \leq T_{\ell+1}^\delta\le T+1 ; 
T_{\ell+1}^\delta- T_\ell^\delta \leq \delta^\xi  \mid \mathcal{F}_{T_\ell^\delta} \right] \\
&\leq \P\left[ t_\ell^\delta(\sqrt{\varepsilon}/3) \leq 
\wedge T+1   ;  T_{\ell+1}^\delta- t_\ell^\delta(\sqrt{\varepsilon}/3) \leq \delta^{\xi}
 \mid \mathcal{F}_{T_\ell^\delta}  \right] \\ 
 &+ 
 \P\left[  \bar{t}_\ell^\delta(i,\sqrt{\varepsilon}/6) \leq t_\ell^\delta(\sqrt{\varepsilon}/3) ;   \bar{t}_\ell^\delta(i,\sqrt{\varepsilon}/6) - T_\ell^\delta 
\leq \delta^{\xi}
 \mid \mathcal{F}_{T_\ell^\delta} \right]\,.
\end{align*}
Let us handle the first term of the previous right hand side 
\begin{align*}
\P&\left[ t_\ell^\delta(\sqrt{\varepsilon}/3) \leq  T_{\ell+1}^\delta\wedge (T+1) ;  T_{\ell+1}^\delta- t_\ell^\delta(\sqrt{\varepsilon}/3) \leq \delta^{\xi}
 \mid \mathcal{F}_{t_\ell^\delta(\sqrt{\varepsilon}/3)}  \right]\\
 &\leq \P\left[\max_j \sup_{t_\ell^\delta(\sqrt{\varepsilon}/3)\leq s \leq (t_\ell^\delta(\sqrt{\varepsilon}/3) + \delta^\xi) 
 \wedge t_\ell^\delta(\sqrt{\varepsilon}/12)\wedge (T+1) }
  | \lambda_j^\delta(s) - \lambda_j^\delta(t_\ell^\delta(\sqrt{\varepsilon}/3))| \geq \frac{\sqrt{\varepsilon}}{24}  
  \mid \mathcal{F}_{t_\ell^\delta(\sqrt{\varepsilon}/3)} \right]\\
  &\leq C \exp(-\frac{c\varepsilon^2}{\delta^{\xi}})
\end{align*}
where we used Lemma \ref{lambda_holder} for the last line (actually the proof since we used the estimate for a fixed $s$). For the second term, the idea is similar
\begin{align*}
 \P&\left[  \bar{t}_\ell^\delta(i,\sqrt{\varepsilon}/6) \leq t_\ell^\delta(\sqrt{\varepsilon}/3) ;   \bar{t}_\ell^\delta(i,\sqrt{\varepsilon}/6) - T_\ell^\delta 
\leq \delta^{\xi}
 \mid \mathcal{F}_{T_\ell^\delta} \right]\\
 &\leq \P\left[\max_{j\neq i} \sup_{T_\ell^\delta \leq s \leq (T_\ell^\delta+\delta^\xi)\wedge 
 \bar{t}_\ell^\delta(i,\sqrt{\varepsilon}/6) \wedge (T+1)}
 |\lambda_j^\delta(s) -\lambda_j^\delta(T_\ell^\delta)| \geq \frac{\sqrt{\varepsilon}}{12}  \mid \mathcal{F}_{T_\ell^\delta} \right] \\
& \leq  C \exp(-\frac{c\varepsilon^2}{\delta^{\xi}})\,,
\end{align*}
by Lemma \ref{lambda_holder}. As for all $\xi>0$, $ \exp(-\frac{c}{\delta^{\xi/4}})\ll \delta^{1-p\beta}$ for small enough $\delta$, the proof is complete.
\qed
\section{Properties of the eigenvalues of $M^\beta_n$}\label{eigenvalues_of_M_n}
In this section, we will study the regularity and boudedness properties
of the eigenvalues of $M^\beta_n$.

\begin{definition}\label{def_lambda_n}
Let $M_0^\beta$ be  a symmetric (resp. Hermitian)
matrix  if $\beta=1$ (resp. $\beta=2$)
with distinct eigenvalues $\lambda_1 < \lambda_2 < \dots < \lambda_d$ and $(M_n^\beta(t))_{t\geq 0}$ be the matrix process defined in Definition \ref{defM}.
For all $t\geq 0$, the ordered eigenvalues of the matrix $M_n^\beta(t)$ will be denoted by $\lambda_1^n(t)\leq \lambda_2^n(t) \leq \dots \leq \lambda_d^n(t).$
\end{definition}
%
The following proposition characterizes the evolution of the process $\lambda^n(t)$ until its first collision time. 
\begin{proposition}\label{repr_eigenvalues}
Let $(\lambda_1^n(t),\dots,\lambda_d^n(t))$ be  the process defined in Definition \ref{def_lambda_n} and set $T_n(1):=\inf\{t\geq 0: \exists i \neq j, \lambda_i^n(t)=\lambda_j^n(t)\}$.
Then, almost surely, the process $(\lambda_1^n(t),\dots,\lambda_d^n(t))$ verifies for every $k\in \N$, the following strict inequality 
\begin{equation}\label{strict_ineq}
\lambda_1^n(k/n)< \lambda_2^n(k/n) < \dots < \lambda_d^n(k/n)\,.
\end{equation} 
In addition, there exist a sequence of Bernoulli random variables $(\epsilon_k^n)_{k\in \N}$ with mean $p$  and a sequence of independent (standard) Brownian motions 
$(b_t^i)_{t\geq 0},\, i \in \{1,\dots,d\}$ also independent of the Bernoulli random variables $(\epsilon_k^n)_{k\in \N}$ such that, the process 
$(\lambda_1^n(t),\dots,\lambda_d^n(t))_{t\geq 0}$ is the re-ordering of the process $(\mu_1^n(t),\dots,\mu_d^n(t))_{t\geq 0}$ defined for $t\geq 0$ by
\begin{equation}\label{defmu}
d\mu_i^n(t) = - \gamma \mu_i^n(t)\,dt + \sqrt{2} db_t^i + \beta\sum_{j\neq i} \frac{\epsilon_t^n}{\mu_i^n(t)-\mu_j^n(t)} \, dt\,.
\end{equation}
with initial conditions in $t=0$ given by $(\mu_1^n(0),\dots,\mu_d^n(0))=(\lambda_1,\dots,\lambda_d)$.
In particular, up to time $T_n(1)$, the process $\lambda^n$ verifies 
\begin{equation*}
d\lambda_i^n(t) = - \gamma \lambda_i^n(t)\,dt + \sqrt{2} db_t^i + \beta\sum_{j\neq i} \frac{\epsilon_t^n}{\lambda_i^n(t)-\lambda_j^n(t)} \, dt\,.
\end{equation*}
\end{proposition}
Remark here that we use the property that $\epsilon^n_t=(\epsilon^n_t)^2$.

\noindent
{\bf Proof.}
Let us show first that for each $k\in \N$ such that $k/n<T_n(1)$, we have almost surely the strict inequality \eqref{strict_ineq}.
We will proceed by induction over $k$. Note that under our assumptions, it is true for $k=0$. Suppose it is true at rank $k$ and let us show it is then true at rank $k+1$.
From Definition \ref{defM}, if the eigenvalues of $M_n^\beta(k/n)$ are denoted as $\lambda_1^n(k/n) < \dots < \lambda_d^n(k/n)$,
then, depending on the value of the Bernoulli random variable $\epsilon_k^n$, the dynamic for $t\in [k/n;(k+1)/n]$ is
\begin{itemize}
\item if $\epsilon_k^n =1$, the process $(\lambda_1^n(t), \dots, \lambda_d^n(t))$ follows the Dyson Brownian motion with initial conditions 
$(\lambda_1^n(k/n), \dots, \lambda_d^n(k/n))$ (see \cite[Theorem 4.3.2]{AGZ}); More precisely, we have  
for $t\in [k/n;(k+1)/n)$  
\begin{equation*}
d\lambda_i^n(t) = - \gamma \lambda_i^n(t)\,dt + \sqrt{2} dW_t^i + \beta\sum_{j\neq i} \frac{dt}{\lambda_i^n(t)-\lambda_j^n(t)}\,. 
\end{equation*}
where the $(W_t^i)_{t\geq 0},i\in \{1,\dots,d\}$ are independent Brownian motions. 
In particular,
 this process is non-colliding in the sense that the $\lambda_i^n(t)$ will almost surely remain strictly ordered for all $t\in [k/n;(k+1)/n)$ 
(see \cite[Theorem 4.3.2]{AGZ}). Thus, we will almost surely have $\lambda_1^n((k+1)/n)<\dots< \lambda_d^n((k+1)/n)$.
\item on the other hand, if $\epsilon_k^n =0$, we need to define a new process $(\mu_1^n(t),\dots,\mu_d^n(t))$ of independent Ornstein�-Uhlenbeck processes 
with initial conditions $(\lambda_1^n(k/n), \dots, \lambda_d^n(k/n))$; More precisely, the evolution for $t\in [k/n;(k+1)/n]$ is given by 
\begin{equation}\label{lambda_is_BM}
d\mu_i^n(t) = -\gamma \mu_i^n(t) dt + \sqrt{2}dB_t^i
\end{equation}
where the Brownian motions $B^i$ are the ones of Definition \ref{defM}. Note that, before time $T_n(1)$, the two processes 
$\lambda^n$ and $\mu^n$ coincide. 
In this case, the $\mu_i^n(t)$ can cross and the ordering can be broken in the interval $[k/n;(k+1)/n]$. However, if crossing for the process $\mu^n$ happen 
before time $t=(k+1)/n$ still we know that 
 $e^{\gamma(k+1)/n} \mu_i^n((k+1)/n)
$ are almost surely distinct. The re-ordering of the $\mu_i^n$ thus always gives
$\lambda_1^n((k+1)/n) < \dots < \lambda_d^n((k+1)/n)\,  a.s.$
\end{itemize}
The induction is complete and proves equality \eqref{strict_ineq} for all $k\in\N$.
We deduce from the above arguments that for $k$ such that $k/n< T_n(1)$, the evolution of  
$\lambda^n(t)$ for $t\in [k/n;(k+1)/n\wedge T_n(1))$ is  
\begin{align*}
d\lambda_i^n(t) = - \gamma \lambda_i^n(t)\,dt + \sqrt{2} (\epsilon_t^n dW_t^i+(1-\epsilon_t^n) dB_t^i) + \beta\sum_{j\neq i} 
\frac{\epsilon_t^n}{\lambda_i^n(t)-\lambda_j^n(t)} \,dt\,.
\end{align*} 
with initial conditions in $t=k/n$ given by $(\lambda_1^n(k/n), \dots, \lambda_d^n(k/n))$.
Let us define the process $b^i$ for $t\geq 0$ by $b_t^i := \int_0^t (\epsilon_s^n dW_s^i+(1-\epsilon_s^n) dB_s^i)$. 
Using the fact that the Brownian motions $(W_t^i)_{t\geq 0}, i \in\{1,\dots, d\} $ are mutually independent and independent of the Brownian 
motions $(B_t^i)_{t\geq 0}, i \in\{1,\dots, d\} $ (also mutually independent), 
it is straightforward to check that the processes $(b_t^i)_{t\geq 0}, i \in\{1,\dots, d\} $ are mutually independent Brownian motions. 
It is also easy to see that, for all $s,t \in [k/n;(k+1)/n]$, the random variables $\epsilon_k^n (W_t^i-W_s^i) + (1-\epsilon_k^n) (B_t^i - B_s^i)$ and 
$\epsilon_k^n$ are independent. Therefore, we deduce that the brownian motions $(b_t^i)_{t\geq 0}, i \in\{1,\dots, d\} $ are independent of the sequence 
$(\epsilon_k^n)_{k\in \N}$.
\qed

The following regularity properties will be useful later on.  
\begin{lemma}\label{regularityM_n}
Let $T<\infty$.
Then there exist constants $C,A_0,c,c',\alpha>0$ which depend only on $T,d$ such that for all $n\in \N$, all 
$A\geq A_0$ and all $\varepsilon >0 $
\begin{align}
\P\left[\max_{1\leq i,j \leq d} \sup_{0\leq t \leq T} |M_n^\beta(t)_{ij} | > A \right] & \leq C \exp(-\alpha A^2) \,,\\
\P\left[\max_{1\leq i,j \leq d} \sup_{0\leq s,t \leq T,\atop |t-s|\leq \delta } |M_n^\beta(t)_{ij} - M_n^\beta(s)_{ij} | > \varepsilon \right] & \leq \frac{c}{\delta} 
\exp(-\frac{\varepsilon^2}{c'\delta})\,.
\end{align} 
\end{lemma}
{\bf Proof.}
Using It\^o's formula, we can check that  
\begin{equation*}
e^{\gamma t} M_n^\beta(t) - e^{\gamma s} M_n^\beta(s) = \int_s^t e^{\gamma s} \left(\epsilon_s^n dH_s^\beta + (1-\epsilon_s^n) \sqrt{2}\sum_{i=1}^d 
\chi_i^n(\frac{[ns]}{n}) dB^i_s\right)\,. 
\end{equation*}
Let us set $\Delta_n(s,t):= e^{\gamma t} M_n^\beta(t) - e^{\gamma s} M_n^\beta(s)$.
The entries of $\Delta_n(s,.)$ are martingales with respect to the filtration of the Brownian motions conditionally to the Bernoulli random variables $(\epsilon_k^n)_{k \in \N}$
(this is due to the independence between the Brownian motions $(B^i_t)_{t\geq 0},(H_t^\beta(ij))_{t\geq 0}, 1\leq i , j \leq d$ and the sequence of Bernoulli random variables 
$(\epsilon_k^n)_{k \in \N}$.
Using the fact that $|\chi_i^n([ns]/n)_{ij}|\leq 1$ for all $i,j$, we can check that there exists a constant $C(d,T)$ which does not depend on $n$ such that for all $n\in \N$
\begin{equation*}
|\langle \Delta_n(s,\cdot)_{ij}, \Delta_n(s,\cdot)_{kl} \rangle_t | \leq C(T,d) |t-s|\,.
\end{equation*}
Let $A >0$, using \cite[corollary H.13]{AGZ}, we have
\begin{align}
&\P\left[\max_{1\leq i,j\leq d}
\sup_{0\leq t \leq T} | (e^{\gamma t} M_n^\beta(t) )_{ij} | > A \right]\nonumber \\
&\leq d^2 \max_{1\le i,j\le d} \P\left[
\sup_{0\leq t \leq T} | (e^{\gamma t} M_n^\beta(t) - M_0^\beta )_{ij}| > A -\max_{i,j} | M_0^\beta(i,j) | \right]\nonumber \\
&= d^2\max_{1\leq i,j\leq d}
\P\left[\sup_{0\leq t \leq T} | \Delta_n(0,t)_{ij}| > A -\max_{i,j} | M_0^\beta(i,j) | \right]\nonumber\\
& \leq d^2\exp\left(-\frac{(A-\max_{i,j}|M_0^\beta(i,j)|)^2}{C(d,T) T}\right).\label{cf}
\end{align}
Similarly,
for any given $s\in [0,T]$,
for  $\varepsilon>0$, using \cite[Corollary H.13]{AGZ}, we have, for each entry $ij$ and for every $\delta >0$:
\begin{align*}
&\P\left[\max_{1\leq i,j\leq d}\sup_{ t\in [s-\delta,s+\delta]} | (e^{\gamma t} M_n^\beta(t) - e^{\gamma s} M_n^\beta(s) )_{ij} | > \varepsilon \right] 
 \leq d^2 \exp\left(-\frac{\varepsilon^2}{2 C \delta}\right).
\end{align*}
and therefore there exists a positive constant $c'$ so that 
\begin{align*}
&\P\left[\max_{1\leq i,j\leq d}
\sup_{0\leq s,t \leq T, \atop |t-s| \leq \delta} |(e^{\gamma t} M_n^\beta(t) - e^{\gamma s} M_n^\beta(s) )_{ij} | > \varepsilon \right] \\
&\leq \sum_{i=1}^{[2T/\delta]+1}
\P\left[\max_{1\le i,j\le d}
\sup_{ | t-\frac{i\delta}{2}| \leq \delta/2} | (e^{\gamma t} M_n^\beta(t) - e^{\gamma i\delta/2} M_n^\beta(i\delta/2) )_{ij} | > \varepsilon/2 \right]\\
& \leq d^2 \frac{2T}{\delta}
\exp\left(-\frac{\varepsilon^2}{c'\delta}\right).\\
\end{align*}
\qed
\begin{lemma}\label{regularity_lambda_n}
Let $T<\infty$.
Then there exist constants $C',A_0,c',c'',\alpha,\epsilon _0>0$ which depend only on $T,d$ such that for all $n\in \N$, all 
$A\geq A_0$ and all $\varepsilon >0 $
\begin{align}
\P\left[\max_{1\leq i \leq d} \sup_{0\leq t \leq T} |\lambda_i^n(t) |  > A \right] & \leq C' \exp(-\alpha A^2) \label{reg_lamb_eq1}\,,  \\
\P\left[\max_{1\leq i \leq d} \sup_{0\leq s, t \leq T, \atop  |t-s|\leq \delta } |\lambda_i^n(t) - \lambda_i^n(s)| > \varepsilon \right] & \leq\frac{c''}{\delta} 
\exp(-\frac{\varepsilon^2}{c'\delta})\,. 
\end{align} 
\end{lemma}
{\bf Proof.}
This lemma is a consequence of Lemma \ref{regularityM_n} and the inequalities
\begin{align}
\max_{1\le k\le d}
|\lambda_k^n(t) - \lambda_k^n(s) | &\leq \left(\sum_{i=1}^d | \lambda_i^n(t) - \lambda_i^n(s) |^2
\right)^{\frac{1}{2}} \nonumber\\
& = 
 \left( \sum_{i,j=1}^d | 
M_n^\beta(t)_{ij} - M_n^\beta(s)_{ij} |^2 \right)^{1/2} \label{HW}\\
&\leq d\max_{1\le i,j \le d}  | 
M_n^\beta(t)_{ij} - M_n^\beta(s)_{ij} | \nonumber
\end{align}
where, for the second inequality, we used \cite[ lemma 2.1.19]{AGZ}
and the fact that the $\lambda^n_i$ are ordered.
\qed

\section{Convergence of the law of the eigenvalues till the first hitting time}\label{conv_until_first_coll}
\begin{proposition}\label{firsttheo} Take $\lambda(0)=(\lambda_1<\lambda_2<\cdots<\lambda_d)$. 
Construct $\mu^n$,   strong
solution of \eqref{defmu}, with the same Brownian motion than $\lambda$, strong
solution of \eqref{theeqlim}, both starting from
$\lambda(0)$. $\lambda^n$ equals $\mu^n$ till $T_n(1)$. For all $T>0$,
we have the following almost sure  convergence 
\begin{equation*}
\lim_{n\ra\infty}\max_{1\leq i\leq d}
\sup_{t\leq T\wedge T_n(1)\wedge \tau^3_\varepsilon}|\lambda_i^n(t)-\lambda_i(t)|
=0\,.
\end{equation*}
As a consequence, 
if we let  $T_1= \inf\{ t > 0, \exists i\neq j, \, \lambda_i(t) = \lambda_j(t)  \}$, we have almost surely 
\begin{equation*}
T_1 \leq \liminf T_n(1)\,.
\end{equation*}
\end{proposition}
We point out that this convergence does not happen on a trivial interval since we have
\begin{rem}\label{int_non_tri}
For any $\eta>0$, there exists  $\tau(\eta)>0$ so that 
\begin{equation*}
\lim_{n\ra\infty} \P\left[  T_n(1)\geq \tau(\eta)\right] \geq 1-\eta\,.
\end{equation*}
\end{rem}

{\it Proof of Remark \ref{int_non_tri}.}
By the same arguments developed in \eqref{HW}, we find that
\begin{eqnarray*}
\P\left[\sup_{t\leq T}\max_{1\leq i\leq d}|\lambda_i^n(t)e^{\gamma t}
-\lambda_i(0)|\geq \e\right]&\leq&
\P\left[\sup_{t\leq T}|\tr((M^n(t)e^{\gamma t}
-M_0)^2)|\geq \e^2\right] \\
&\leq &d^2 \exp(-\frac{\e^2}{2C(d,T)T})\,.
\end{eqnarray*}
But since also the $\lambda_i^n$ are uniformly bounded with high probability,
we can choose for any $\eta>0$ the parameter $T$ small enough so that
\begin{equation*}
\P\left[\max_{1\leq i\leq d}\sup_{t\leq T}|\lambda_i^n(t)-\lambda_i(0)|\ge \min_{1\leq i\leq d}|\lambda_i-\lambda_{i+1}|/3\right]\leq \eta
\end{equation*}
This implies that $P(T_n(1)\le T)\le \eta$.
\qed

{\bf Proof of Proposition \ref{firsttheo}}
Using It\^o's formula, we can compute 
\begin{align}\label{eds_sum_square}
\sum_{i=1}^{d} \left(\lambda_i^n(t)-\lambda_i(t)\right)^2 &= -2\gamma \int_0^t \sum_{i=1}^{d} \left(\lambda_i^n(s)-\lambda_i(s)\right)^2 ds \\ \notag
&+ 2 \beta \int_0^t \e_s^n \sum_{i=1}^{d}\sum_{j\not = i} \left(\lambda_i^n(s)-\lambda_i(s)\right) 
\left(\frac{1}{\lambda_i^n(s)-\lambda_j^n(s)}-\frac{1}{\lambda_i(s)-\lambda_j(s)}\right) ds \\ \notag
&+2\beta \int_0^t (\e_s^n-p) \sum_{i=1}^{d} \sum_{j \not = i} \frac{\lambda_i^n(s)-\lambda_i(s)}{\lambda_i(s)-\lambda_j(s)} ds \,.
\end{align}
By the same argument as in \eqref{astuce}
the second term in the right hand side is non positive.
Thus using equations \ref{eds_sum_square}, we find for $t \leq T_n(1)$ 
\begin{align*}
\sum_{i=1}^{d} \left(\lambda_i^n(t)-\lambda_i(t)\right)^2 \leq 2\beta \int_0^t (\e_s^n-p) \sum_{i=1}^{d} \sum_{j \not = i} \frac{\lambda_i^n(s)-\lambda_i(s)}{\lambda_i(s)-\lambda_j(s)} ds
:= R_n(t)\,.
\end{align*}
We next prove that
 \begin{equation}\label{R_n}
\lim_{n\ra\infty}
\sup_{0\leq t\leq T\wedge\tau^3_\varepsilon} R_n(t)=0 \quad a.s.\end{equation}
Write $R_n(t)$ as $R_n(t) = P_n(t) + Q_n(t)$ where
\begin{align*}
P_n(t)&:= \int_0^t (\e_s^n-p) \sum_{i=1}^{d} \sum_{j\not=i} \frac{\lambda_i^n([ns]/n)-\lambda_i(s)}{\lambda_i(s)-\lambda_j(s)}ds\,, \\
Q_n(t)&:= \int_0^t (\e_s^n-p) \sum_{i=1}^{d} \sum_{j\not=i} \frac{\lambda_i^n(s)-\lambda_i^n([ns]/n)}{\lambda_i(s)-\lambda_j(s)}ds\,.
\end{align*}
We first handle the convergence of $Q_n(t)$.
Set  $\Omega_1=\{\sup_{|s-t|\le 1/n\atop t\le T}\max_{1\le i\le d}|\lambda^n_i(t)-\lambda^n_i(s)|\leq n^{-1/2+\e}\}$.
On the event $\Omega_1$, we have
\begin{align*}
|Q_n(t)|\le n^{-1/2+\e} \sum_{i=1}^{d} \sum_{j\not=i} \int_0^t \frac{ds}{\mid\lambda_i(s)-\lambda_j(s)\mid} \,.
\end{align*}
Following \eqref{HW}, we know that
$$P(\Omega_1^c)\le c e^{-cn^{2\e}} \,.$$
We thus deduce from  Lemma \ref{lemCepa} that
\begin{align*}
\P\left[\sup_{t\leq T} |Q_n(t)| > \delta \right] &\leq \P\left[ \sum_{i=1}^{d} \sum_{j\not=i} \int_0^T \frac{ds}{\mid\lambda_i(s)-\lambda_j(s)\mid} >\delta n^{1/2-\e}\right] + 
\P\left[\Omega_1^c\right] \\
&\leq c\, e^{-c \,\delta^2\, n^{1-2\epsilon}} + c\, e^{-c\,n^{2\e}}\,.
\end{align*}
Hence, Borel Cantelli's Lemma insures the almost sure convergence of 
$Q_n$ to zero.
We now turn to the convergence of $P_n(t)$. Let $\eta >0$ small and 
 write
$$P_n(t)=-\frac{d(d-1)}{2} \int_0^t (\e^n_s-p) ds+ \tilde P_n(t)$$
with
$$ \tilde P_n(t)= \int_0^t (\e_s^n-p) \sum_{i=1}^{d} \sum_{j<i}
 \frac{\lambda_i^n([ns]/n)-\lambda_j^n([ns]/n)}{\lambda_i(s)-\lambda_j(s)}ds\,.$$
The process $\int_0^t (\e^n_s-p) ds$ is a martingale and by
Azuma-Hoeffding inequality, for any $\delta>0$
$$\P\left(\max_{t\le T}|\int_0^t (\e^n_s-p) ds|\ge \delta\right)\le 
2\exp(-\frac{\delta^2 n}{2})\,.$$
We now use the independence between the brownian motions $(b^i_t)_{0\leq t \leq T}, i =1,\dots,d$
and the Bernoulli random variables $\e_k^n,k=1,\dots,[nT]$. 
Conditionally on the $(b^i_t)_{0\leq t \leq T}, i =1,\dots,d$, 
the processes $\lambda_i(t), i=1,\dots,d$ are deterministic and 
the process $\tilde P_n$ is a martingale with respect to the filtration of the $\e_k^n$.  We let 
$$A^n_k =\sum_{i=1}^{d} \sum_{j<i}\int_{k/n}^{k+1/n}
 \frac{\lambda_i^n([ns]/n)-\lambda_j^n([ns]/n)}{\lambda_i(s)-\lambda_j(s)}ds.$$
By Lemma \ref{lambda_holder} and Lemma \ref{regularity_lambda_n}, the set
$$\Omega=\{\sup_{k\le n T\wedge \tau^3_\e} | A^n_k|
\le n^{-1/8}\} $$
has probability larger than  $1-e^{-c n^{1/16}}$. Moreover, by martingale property   it is easy to see that for all $\lambda\ge 0$,
$$\bE[ 1_{\Omega}e^{\lambda \tilde P_n(k/n) -\frac{1}{2}\lambda^2 \sum_{\ell=0}^{k-1}
(A^n_{k/n})^2}]\le 1\,.$$ Taking $\lambda=n^{1/16}$,
since on $\Omega$, $- n^{1/16} |A^n_k| +n^{1/8} |A^n_k|^2/2\le 0$,
Tchebychev's inequality yields
$$\bP\left( \{ |\tilde P_n(k/n\wedge \tau^3_\varepsilon)|
\ge n^{-1/16}(\sum_{\ell=0}^{[Tn]} |A^n_k| +t)
\}\cap\Omega\right)\le e^{-t}$$
As by Lemma \ref{lemCepa}, $\sum_{\ell=0}^{[Tn]} |A^n_k|$
is bounded by  $n^{1/32}$ with probability greater than $1-e^{-n^{1/16}}$
we conclude that
$$\bP\left( |\tilde P_n(k/n \wedge \tau^3_\varepsilon)|\ge n^{-1/32}\right)
\le C e^{-n^{1/32}}\,.$$
The uniform estimate is obtained easily by  controlling  the increments of $\tilde P_n$ in between the times
$k/n, k\le [nT]$ by   $\sup_{k\le [nT]}|A^n_k|$ 
which we have already bounded.

\qed
\section{Proof of Theorem \ref{main}.}\label{end_proof_main}
\subsection{Non colliding case $p\beta\geq 1$}\label{pbetagrand}

It is straightforward to deduce 
Theorem \ref{main} when $p\beta\geq 1$.
Indeed if $\beta p\geq 1$ we know that there are no collisions for the limiting process
and more precisely, see e.g \cite[p. 252]{AGZ},
\begin{equation*}
\P(\tau_\varepsilon^2 \leq T)\leq c(\lambda_0) T/|\log \varepsilon|
\end{equation*}
with some finite constant $c(\lambda_0)$ which only depends on the spacings of the eigenvalues
at the initial time.
This implies in particular that
\begin{equation*}
\lim_{\varepsilon\ra 0} \lim_{n\ra\infty} \P(T_\varepsilon^n \leq T)=0
\end{equation*}
from which we easily deduce Theorem \ref{main} from Proposition \ref{firsttheo}.

\subsection{Colliding case $p\beta<1$}\label{auxiliaryprocesssec}

We now define the process $(\lambda_i^{n,\delta}(t))_{t\geq 0}$ which will depend on the sequence $(T_\ell^\delta)_{\ell\in \N}$ defined in Definition \ref{deflimprocdelta}. To unify notations, 
set $T_1^\delta:=T_1$ and $T_n^\delta(1):=T_n(1)$. 
\begin{definition}\label{def_lambda_n,delta}
For $t< T_1^\delta$, set $\lambda_i^{n,\delta}(t):=\lambda_i^n(t)$. 
For time $t>T_1^\delta$, we define the process recursively by setting for each $\ell \geq 1$, 
$\lambda_i^{n,\delta}(T_\ell^\delta)= \lambda_i^{n,\delta}(T_\ell^\delta-)+ i \delta$ for all $i \in\{1,\dots,d\}$ and for $t>T_\ell^\delta$,
the process $\lambda_i^{n,\delta}$ is defined up to time $T_{\ell+1}^\delta$ by ordering the process $(\mu_1^{n,\delta}(t),\dots,\mu_d^{n,\delta}(t))_{T_\ell^\delta\leq 
t \leq T_{\ell+1}^\delta}$ which is defined for $t\geq T_\ell^\delta$ as  
\begin{equation}\label{dynamic_lambda_n,delta}
d\mu_i^{n,\delta}(t) = - \gamma \mu_i^{n,\delta}(t)\,dt + \sqrt{2} db_t^i + \beta\sum_{j\neq i} \frac{\epsilon_t^n}{\mu_i^{n,\delta}(t)-\mu_j^{n,\delta}(t)} \, dt\,.
\end{equation}
with initial conditions in $t= T_\ell^\delta$ given by $(\lambda_1^{n,\delta}(T_\ell^\delta),\dots,\lambda_d^{n,\delta}(T_\ell^\delta))$.
\end{definition}

\begin{lemma}\label{lim_ndelta}
Let $T<\infty$ and $\delta >0$.
We have the following convergence in probability, for all $\ell\in \N$,
\begin{align*}
\lim_{n\to \infty} \max_{1\leq i \leq d} \sup_{0 \leq t
 \leq T_\ell^\delta\wedge  T} | \lambda_i^\delta(t) - \lambda_i^{n,\delta}(t)| = 0 \,.
\end{align*}
In particular, for every $\ell$, if $T_n^\delta$ is the first collision time for $\lambda^{n,\delta}$
after $T^\delta_{\ell -1}$,
$$T_\ell^\delta\wedge T \leq \liminf T_n^\delta(\ell)\wedge T \quad a.s. $$
\end{lemma}
{\bf Proof}
Again, we prove this Lemma by induction over $\ell$.

$\bullet$ We begin with the case $\ell=1$. Proposition \ref{firsttheo} yields that 
the random variable $ \max_{1\leq i \leq d} \sup_{0\leq t \leq T_n(1)\wedge T} |\lambda_i(t) - \lambda_i^n(t) | = 0$ converges to $0$ 
in probability as by Lemma \ref{multiple_collisions},  $P(\tau_\varepsilon^3\ge T)$ goes to one as $\varepsilon$ vanishes.  
Since we have the almost sure inequality $T_1^\delta \leq\liminf T_n^\delta(1)$, the continuity of the $\lambda_i,1\leq i\leq d$, the regularity 
property of the $\lambda_i^n$ given by Lemma \ref{regularity_lambda_n}, Lemma \ref{lambda_holder} and Proposition \ref{firsttheo}, we can check that since before $T_1^\delta$
$\lambda_i^\delta=\lambda_i$ and $\lambda^{n,\delta}_i=\lambda_i^n$, 
if $T_n^\delta(1)<T_1^\delta\wedge T$, 
\begin{align}\label{small_interval}
\max_{1\leq i \leq d}& \sup_{ T_n^\delta(1)\leq t< T_1^\delta\wedge T} | \lambda_i^\delta(t) - \lambda_i^{n,\delta}(t)| \\
&\leq \max_{1\leq i \leq d} \sup_{ T_n^\delta(1)\leq t< T_1^\delta\wedge T}\{ |\lambda_i^n(t)-\lambda_i^n(T_n^\delta(1))| 
+ |\lambda_i(t)-\lambda_i(T_n^\delta(1))|\} \\ 
&+ |\lambda_i^{n}(T_n^\delta(1))-\lambda_i(T_n^\delta(1))|\notag
\end{align}
goes to zero in probability, when $n$ goes to infinity.

$\bullet$ Suppose the property is true for $\ell$ and let us show that it is then true for $\ell+1$.
By the same argument as in the proof of Proposition \ref{firsttheo}, we can show that, for all $t\in[T_\ell^\delta; T_n^\delta(\ell+1) \wedge T_{\ell+1}^\delta]$, we have
\begin{align}\label{eq_limndelta}
\sum_{i=1}^{d}& \left(\lambda_i^{n,\delta}-\lambda_i^{\delta}\right)^2(t) \leq 
\sum_{i=1}^{d} \left(\lambda_i^{n,\delta} -\lambda_i^{\delta}\right)^2(T_\ell^\delta) \\
&+ 2\beta \int_{T_\ell^\delta}^{t} (\epsilon_s^n-p) \sum_{i=1}^{d} \sum_{j \not = i} \frac{\lambda_i^{n,\delta}(s)-\lambda_i^\delta(s)}{\lambda_i^{\delta}(s)-\lambda_j^\delta(s)} ds. \notag
\end{align}
The same proof as in Proposition \ref{firsttheo} shows that, 
if $\tau^{3,\ell}_\e$ is the stopping time $\tau^3_\ell$ for the process
$\lambda^\delta(t), t\ge T_\ell^\delta$,

\begin{equation}
\lim_{n\to \infty} \sup_{t\in [T_\ell^\delta;T_n^\delta(\ell+1) \wedge T_{\ell+1}^\delta\wedge \tau^{3,\ell}_\e]} 
\int_{T_\ell^\delta}^{t} (\e_s^n-p) \sum_{i=1}^{d} \sum_{j \not = i} \frac{\lambda_i^{n,\delta}(s)-\lambda_i^\delta(s)}{\lambda_i^{\delta}(s)-\lambda_j^\delta(s)} ds=0 \quad a.s.
\end{equation}
Thus, because of \eqref{eq_limndelta}, the following convergence in holds
\begin{equation}\label{conv_biginterval}
\lim_{n\to \infty} \max_i \sup_{t\in [T_\ell^\delta;T_n^\delta(\ell+1) \wedge T_{\ell+1}^\delta \wedge\tau^3_\e]} |\lambda_i^{n,\delta}(t)-\lambda_i^{\delta}(t)| =0\quad a.s\,.
\end{equation}
Because of \eqref{conv_biginterval}, we have $T_{\ell+1}^\delta\wedge\tau_\e^3 \leq\liminf_{n\to \infty} T_n^\delta(\ell+1)\wedge \tau^3_\e$.
Since the probability that $\tau^3_\e$ is larger than $ T$ goes
to one as $\e$ vanishes,  we 
can show as in \eqref{small_interval} (note that Lemma \ref{regularity_lambda_n}, Lemma \ref{lambda_holder} and Proposition \ref{firsttheo} extend to 
$\{\lambda^{n,\delta}_{t},\lambda^\delta_t,  t\ge T^\delta_\ell\}$)
that in probability, 
\begin{align*}
\lim_{n\to\infty} \max_{1\leq i \leq d}& \sup_{ T_n^\delta(\ell+1) \leq t\leq T_{\ell+1}^\delta} | \lambda_i^\delta(t) - \lambda_i^{n,\delta}(t)| =0.
\end{align*}
The property at rank $\ell+1$ is established. The Lemma is proved. 
 \qed

\begin{lemma}\label{comp_discret}
There exists a constant $c >0$ such that for all $L \in \N$, we have the following almost sure estimate
\begin{align*}
 \max_{1\leq j \leq d} \sup_{0 \leq t \leq T_{L}^\delta} 
|\lambda_j^{n,\delta}(t) - \lambda_j^n(t)| \leq \delta\, L \, \sqrt{c} \,.
\end{align*}
\end{lemma}

{\bf Proof.} Note that the estimate is striaghtforward on $[0, T_1^\delta]$.
We then proceed by induction on the time intervals $[T_\ell^\delta,T_{\ell+1}^\delta]$ as 
 in the proof  of Lemma \ref{lemjk}
until the first collision time
$$t_1:= \inf\{ t \geq T_\ell^\delta \,: \exists i, \,\lambda_i^n(t) = \lambda_{i-1}^n(t)\, \text{or} \, \lambda_i^{n,\delta}(t) = \lambda_{i-1}^{n,\delta}(t) \}\,.$$

We next claim that, at a given time, almost surely the eigenvalues $\lambda^n$ are different. Indeed, this is clear if the eigenvalues follows Brownian motion and even more  when they follow  Dyson Brownian motion. Moreover the probability that  more than two eigenvalues collide at some time vanishes. Indeed, this can only happen if the eigenvalues follow the Brownian motion. But the probability
that 3 Brownian motions collide vanishes and hence the result.	 

Hence, there are almost surely
at most two eigenvalues which can collide. Hence, 
let $i(t_1)$ be the unique integer in $\{1,\dots,d\}$ such that $\lambda_i^n(t_1)=\lambda_{i-1}^n(t_1)$ (respectively $\lambda_i^{n,\delta}(t_1)=
\lambda_j^{n,\delta}(t_1)$) and let $\tau_1=([nt_1]+1)/n$.
Notice that, for $t \in[[nt_1]/n; ([nt_1]+1)/n )$, we necessarily have $\e_t^n = 0$.
Let $\mu_i^{n,\delta}$ and $\mu_i^n$ for $i\in \{1,\dots,d \}$ be the processes such that for $t \in[t_1; \tau_1]$
\begin{align*}
d\mu_i^{n,\delta}(t) &= - \gamma \mu_i^{n,\delta}(t) dt +\sqrt{2} db^i_t\, \\
d\mu_i^{n}(t) &= - \gamma \mu_i^{n}(t) dt +\sqrt{2} db^i_t
\end{align*}
with initial conditions at $t=t_1$ respectively given by $\mu^{n,\delta}(t_1)=\lambda^{n,\delta}(t_1)$ and 
$\mu^{n}(t_1)=\lambda^{n}(t_1)$. 
We know that the $\lambda_i^{n,\delta}$, respectively the $\lambda_i^n$, are just a re-ordering of the processes $\mu_i^{n,\delta}$ and $\mu_i^n$

By definition, for  $t\in[t_1;\tau_1]$, we find that :
\begin{align*}
(\mu_j^{n,\delta} - \mu_j^n)(t) &= e^{-\gamma(t-t_1)}
(\mu_j^{n,\delta} - \mu_j^n)(t_1)\,.
\end{align*}
As a consequence, we deduce that
\begin{equation*}
\sum_{j=1}^d (\mu_j^{n,\delta}-\mu_j^n)^2(t)\leq \sum_{j=1}^d (\lambda_j^{n,\delta}-\lambda_j^n)^2(t_1)\,.
\end{equation*}
Moreover, as the $\lambda$'s are ordered but the set of the values of
the $\lambda$'s and the $\mu$'s are the same, using for instance \cite[ lemma 2.1.19]{AGZ}, we have that
\begin{equation*}
 \sum_{j=1}^d (\lambda_j^{n,\delta}-\lambda_j^n)^2(t)\leq \sum_{j=1}^d (\mu_j^{n,\delta}-\mu_j^n)^2(t)\,.
\end{equation*} 

Gathering the above inequalities, we have shown that
 \begin{align*}
\sup_{t\in [0,\tau_1]}
\sum_{j=1}^d(\lambda_j^{n,\delta} - \lambda_j^n)^2(t) \leq \sum_{j=1}^d (\lambda_j^{n,\delta}-\lambda_j^n)^2(T_\ell^\delta) \,.
\end{align*}
We can continue inductively until we  reach the time $T_{\ell+1}^\delta$ to finish the proof.

\section{Asymptotic properties of the eigenvectors}\label{section_eigenvectors}

Recall that $w_{ij}^\beta, i< j $ are real (respectively complex) standard Brownian motions if $\beta=1$ (resp. $\beta=2$) with quadratic variation $\beta t$ and that we also set  
for $i<j$, $w_{ji}^\beta:=\bar{w}_{ij}^\beta$.  
In addition we also defined the skew Hermitian matrix $R^\beta=-(R^\beta)^*$ by setting for $i < j$,
\begin{equation*}
dR_{ij}^\beta(t) = \frac{dw_{ij}^\beta(t)}{\lambda_i^n(t)-\lambda_j^n(t)}, \quad R_{ij}^\beta(0)=0 \,.
\end{equation*}

{\it Proof of Proposition \ref{good_O_n}}

It is classical to check that the unique strong solution of the stochastic differential equation 
\begin{equation}\label{eq_O_n}
dO_n^\beta(t) = \epsilon_t^n O_n^\beta(t) dR^\beta(t) - \frac{\epsilon_t^n}{2} O_n^\beta(t) d\langle(R^\beta)^*,R^\beta\rangle_t \,,
\end{equation}
with initial condition $O_n^\beta(0):=O^\beta(0)$ (defined at the end of Section \ref{intro}),
is in the space $\mathcal{O}^\beta_d$ for all time $t$ (see e.g. \cite[Lemma 4.3.4]{AGZ}) 
and is such that, with $\Delta_n^\beta(t)$ being the diagonal matrix of the ordered (as in \eqref{order_eigenvalues_def}) eigenvalues of $M_n^\beta(t)$, 
we have  
\begin{equation*}
O_n^\beta(t) \Delta_n^\beta(t) O_n^\beta(t)^* \stackrel{law}{=} M_n^\beta(t)\,.
\end{equation*}
The law of the continuous process $O_n^\beta$ is uniquely determined as the unique strong solution of \eqref{eq_O_n}. 
\qed
 
One can thus define the eigenvectors of $M_n^\beta(t)$, denoted as $\phi^n_i(t)$, so that they satisfy the stochastic differential system  
\begin{equation}\label{eqev}
d\phi^n_i (t)=\epsilon^n_t   \sum_{j\neq i}\frac{dw_{ij}^\beta(t)}{\lambda^n_i(t)-\lambda^n_j(t)} \phi^n_j(t)
-\frac{\epsilon^n_t}{2} \sum_{j\neq i} \frac{\beta}{(\lambda^n_i(t)-\lambda^n_j(t))^2} dt \phi^n_i(t)
\end{equation}
where $w_{ij}^\beta, i< j$ is a family of i.i.d. Brownian motions (on $\R$ if $\beta=1$, $\C$ if $\beta=2$), independent of the eigenvalues $\lambda_i^n, 1\leq i \leq d$.

{\it Proof of Theorem \ref{theo_vectors}}

This proof is classical and uses the theory of stability for stochastic differential equations.

For $\eta>0$ fixed, we  deduce from Proposition \ref{firsttheo} 
and Lemma \ref{multiple_collisions} that
the process $(\lambda_1^n(t),\dots,\lambda_d^n(t))$ 
converges almost surely 
in the space of continuous functions $\mathcal{C}([0;(T_1-\eta)\wedge T],\R^d)$ (respectively $\C^d$) if $\beta=1$ (resp. $\beta=2$) endowed with the uniform 
norm towards $(\lambda_1(t),\dots,\lambda_d(t))_{0\leq t \leq (T_1-\eta)\wedge T}$ where the $\lambda_i$'s are the unique strong solutions of \eqref{theeqlim} (with the same Brownian
motions $b^i$) and where 
$T_1$ is the first collision time of the $\lambda_i,1\leq i \leq d$. In the sequel we will work conditionally to the $(\lambda_i^n,\lambda_i)$'s satisfying the above convergence.

Define for $i\neq j$ the processes $w_{ij}^{\beta,n}$ by setting 
\begin{equation}
w_{ij}^{\beta,n}(t) = \int_0^t \epsilon_s^n dw_{ij}^\beta(s) \,.
\end{equation}
Note that the quadratic variation of this continuous martingale converges almost surely towards $\beta p t$ so that by Rebolledo's theorem $(w_{ij}^{\beta,n}, i<j)$ converges towards $(\sqrt{p} w^\beta_{ij}, i<j)$.

Moreover, if $T_1^\epsilon$ is the first time at which two eigenvalues 
are at distance less than $\epsilon$, the drift coefficients being bounded, we see, with a proof similar to the proof of Proposition \ref{firsttheo}, that for $i\neq j$
$$\int_0^{t\wedge T_1^\epsilon}
 \frac{\epsilon^n_s}{(\lambda_i^n-\lambda_{j}^n)^2(s)}ds $$
converges  towards $p\int_0^{t\wedge T_1^\epsilon} (\lambda_i(s)-\lambda_j(s))^{-2}ds $ uniformly almost surely.  Since $T_1^\epsilon$ converges towards $T_1$ as $\epsilon$ goes to zero, the convergence holds till $(T_1-\eta)\wedge T$ for any $\eta>0$.

Gathering the above arguments, the result follows from \cite[Theorem 6.9, p. 578]{Jacod}.
\qed

We now turn to the analysis of the behavior of the 
columns $\phi_i(t)$ of the matrix $O^\beta(t)$ when $t \ra T_1$ with $t<T_1$. Those vectors $\phi_i(t)$ form an orthonormal basis of $\R^d$ (respectively $\C^d$) 
if $\beta=1$ (resp. $\beta=2$) and it is easy to check that they verify the following stochastic differential system
\begin{equation}\label{eqevlim}
d\phi_i (t) =  \sum_{j\neq i}\frac{\sqrt{p} }{\lambda_i(t)-\lambda_j(t)} dw_{ij}^\beta(t)\phi_j(t)
-\frac{p\beta}{2} \sum_{j\neq i} \frac{dt}{(\lambda_i(t)-\lambda_j(t))^2} \phi_i(t) \,.
\end{equation}

In the following of this section, we will denote by $i^*$ the unique (because of Lemma \ref{unicity_i}) index such that $\lambda_{i^*}(T_1)=\lambda_{i^*-1}(T_1)$.

The main issue we meet at this point in the presence of collisions (that {\it will} occur if $p\beta <1$; see \cite{Cepa}) lies in the divergence of the integral 
$\ref{div_integral}$ that we now prove.


We now describe the behavior of the $d-2$ vectors $\phi_j(t), j\neq i^*, i^*-1$ just before the first collision time $T_1$. 


{\it Proof of the first statement of Proposition \ref{limit_phi_T_1}}

We will denote by $\phi_{j\ell}(t)$ the $\ell$-th entry of the $d$-dimensional vector $\phi_j(t)$.
For $0\leq t<T_1$, we have 
\begin{equation}\label{phijC0}
d\phi_j(t) = \sum_{k \neq j} \frac{\sqrt{p}}{\lambda_j(t)-\lambda_k(t)} dw_{jk}^\beta(t) \phi_k(t) - \frac{p}{2} \sum_{k\neq j} \frac{\beta}{(\lambda_j-\lambda_k)^2} \phi_j(t) dt\,.
\end{equation}
We recall from section \ref{estimate_collisions} that there are no multiple collisions nor two collisions at the same time for the system $(\lambda_1(t),\lambda_2(t),\dots,\lambda_d(t))_{0\leq t\leq T_1}$
verifying \eqref{theeqlim}, and therefore we may assume without loss of generality  that  for $j\neq i^*,i^*-1$, every diffusions and drift terms of \eqref{phijC0} remains almost surely bounded for $t\in [0;T_1]$. 
To prove the lemma, we just need to prove that almost surely
\begin{equation*}
\lim_{s\rightarrow T_1; \atop s <T_1} \sup_{s\leq t < T_1} \|\phi_j(t)-\phi_j(s)\|_2 = 0\,.
\end{equation*}
The drift terms appearing in \eqref{phijC0} are obvious to deal with since $1/(\lambda_j-\lambda_k)(t)$ is bounded in the vicinity of $T_1$ and that $|\phi_{j\ell}(t)|\leq 1$ for all $t<T_1$. 
For the diffusion terms, we have for every $\ell\in \{1,\dots,d\}$ and for every $s\in[0;T_1]$ the following estimate
\begin{equation*}
\P\left[\sup_{s\leq t< T_1} | \int_s^t \sum_{k \neq j} \frac{\sqrt{p}}{\lambda_j(u)-\lambda_k(u)} dw_{jk}^\beta(u) \phi_{k\ell}(u) | > \eta \right]
\leq \exp(-\frac{\eta^2}{2\beta p (d-1) M (T_1-s)})\,,
\end{equation*}
where $M = \sup_{t\in [0;T_1]} \max_{k\neq j} \frac{1}{(\lambda_j-\lambda_k)^2(t)}$. Using the Borel-Cantelli Lemma, we deduce the result.
\qed

For $\delta>0$, we want to define a process $(\tilde{\phi}_1(t),\tilde{\phi}_2(t),\dots,\tilde{\phi}_d(t))_{T_1-\delta \leq t < T_1}$ that will be a good approximation of the process 
$(\phi_1(t),\phi_2(t),\dots,\phi_d(t))_{T_1-\delta \leq t < T_1}$ on the time interval $[T_1-\delta;T_1]$.
Hence for $j\neq i^*, i^*-1$, we set $\tilde{\phi}_j(t)=\tilde{\phi}_j$ (the vectors do not depend of time). It remains to define the evolution for 
$(\tilde{\phi}_{i^*-1}(t),\tilde{\phi}_{i^*}(t))$ that will depend of time $t$.
 
Let $V$ be the $(d-2)$-dimensional subspace spanned by the orthonormal family $\{\tilde\phi_j; j \neq i^*,i^*-1\}$ and $W$ 
its orthogonal 
complement in $\R^d$.
Let us define the``diffusive orthonormal basis'' in the space $W$ that will describe the evolution of the two vectors $(\tilde{\phi}_{i^*-1}(t),\tilde{\phi}_{i^*}(t))$ on the interval 
$[T_1-\delta;T_1]$ (up to the initial conditions at time $t=T_1-\delta$ we will explicit later).
\begin{lemma}\label{def_tilde_phi_i_star}
Let $\delta >0$ and $(u,v)$ an orthonormal basis of the two-dimensional subspace $W$. We consider the following stochastic differential system
\begin{align}\label{sds_tilde_phi}
d\tilde{\phi}_{i^*}(t) &= \frac{\sqrt{p}}{(\lambda_{i^*}-\lambda_{i^*-1})(t)} dw_{i^*-1,i^*}^\beta(t) \,\,\tilde{\phi}_{i^*-1}(t) - 
\frac{p\beta}{2} \frac{dt}{(\lambda_{i^*}-\lambda_{i^*-1})^2(t)} \tilde{\phi}_{i^*}(t)\,, \\ 
d\tilde{\phi}_{i^*-1}(t) &= -\frac{\sqrt{p}}{(\lambda_{i^*}-\lambda_{i^*-1})(t)} d\bar{w}_{i^*-1,i^*}^\beta(t) \,\, \tilde{\phi}_{i^*}(t) - 
\frac{p\beta}{2} \frac{dt}{(\lambda_{i^*}-\lambda_{i^*-1})^2(t)} \tilde{\phi}_{i^*-1}(t)\notag\,
\end{align}
with initial conditions $(\tilde{\phi}_{i^*-1}(T_1-\delta),\tilde\phi_{i^*}(T_1-\delta))=(u,v)$.

This stochastic differential system has a unique strong solution defined on the interval $[T_1-\delta;T_1)$ such that 
for each $t\in [T_1-\delta;T_1)$, $\{\tilde{\phi}_{i^*-1}(t),\tilde{\phi}_{i^*}(t)\}$ is an orthonormal basis of $W$. 
\end{lemma} 
\begin{proof}
For all $\epsilon >0$,  the function $t\rightarrow 1/(\lambda_{i^*}-\lambda_{i^*-1})(t)$ is bounded  on the interval $[T_1-\delta; T_1^\epsilon]$ 
and therefore there is a unique 
strong solution to the stochastic differential system \eqref{sds_tilde_phi}
till the time $T_1^\epsilon$ where $|\lambda_{i^*}-\lambda_{i^*-1}|<\epsilon$
as it is driven by bounded linear drifts. As $T_1^\epsilon$ grows to $T_1$
the proof is complete.

To show that for all $t\in [T_1-\delta;T_1)$ the family $\{\tilde{\phi}_{i^*-1}(t),\tilde{\phi}_{i^*}(t)\}$ is an orthonormal basis of $W$, we proceed along the same line 
as in the proof of \cite[Lemma 4.3.4]{AGZ}.  
\end{proof}

In the following lemma, we show that we can choose a constant $\delta >0$ small enough and an initial condition $(u,v)\in W$ such that the 
processes $(\tilde{\phi}_1(t),\dots,\tilde{\phi}_1(t))_{t\in[T_1-\delta;T_1)}$ defined by Lemma \ref{def_tilde_phi_i_star} 
is indeed a good approximation of the process $(\phi_1(t),\dots,\phi_d(t))_{t\in[T_1-\delta;T_1)}$. The advantage of the 
process $(\tilde{\phi}_1(t),\dots,\tilde{\phi}_1(t))_{t\in[T_1-\delta;T_1)}$
is that it is simpler to study in the vicinity of $T_1$ (see Lemma \ref{study_tilde_phi} below). 

\begin{lemma}\label{comp_phi_phi_tilde}
Let $\eta>0$ and $\kappa>0$.
Then there exists an orthonormal basis $(u,v)$ of $W$ and $\delta >0$ small enough such that if we denote by $(\tilde{\phi}_{i^*-1}(t),\tilde{\phi}_{i^*}(t))_{t\in [T_1-\delta;T_1)}$ 
the unique strong solution
of the stochastic differential system \eqref{sds_tilde_phi} with initial conditions given in $t_0=T_1-\delta$ by $(\tilde{\phi}_{i^*-1}(t_0),\tilde\phi_{i^*}(t_0))=(u,v)$, 
we have
\begin{align*}
\P\left(\sup_{t\in [t_0;T_1)}||\phi_{i^*}(t)-\tilde{\phi}_{i^*}(t)||_2^2 + ||\phi_{i^*-1}(t)-\tilde{\phi}_{i^*-1}(t)||_2^2 \geq \eta\right)\le \kappa\,.
\end{align*}
\end{lemma}
\begin{proof}
Using It\^o's formula, we find\footnote{Note that all the diverging terms in $T_1$ cancel in this expression.} for all $t\in [t_0;T_1)$,
\begin{align}\label{error_term_norm2_phi_tildephi}
||\phi_{i^*}(t)-\tilde{\phi}_{i^*}(t)||_2^2 &+ ||\phi_{i^*-1}(t)-\tilde{\phi}_{i^*-1}(t)||_2^2 = ||\phi_{i^*}(t_0)-u||_2^2 + ||\phi_{i^*-1}(t_0)-v||_2^2\notag \\
&-2 \int_{t_0}^t \sum_{i\in \{i^*,i^*-1\}} \sum_{j\neq i^*,i^*-1} \frac{\sqrt{p}}{(\lambda_{i}-\lambda_j)(s)} dw_{ij}^\beta(s) \langle \tilde{\phi}_{i}(s),\phi_j(s)\rangle\,.
\end{align}
 As for $i\in \{i^*,i^*-1\}$ and $j\not\in \{i^*,i^*-1\}$ the terms $1/(\lambda_i-\lambda_j)^2(t)$ have almost surely a finite integral with respect to 
Lebesgue measure on the interval $[t_0;T_1)$ (in fact those terms are almost surely bounded as the corresponding particles remain at finite distance), 
the quadratic variation of the last term is of order $\delta$ and therefore
is smaller than $\eta/2$ with probability greater that $1-\kappa$ for $\delta$ small enough.

It remains to check that we can choose $(u,v)$ an orthonormal basis of $W$ and $\delta>0$ such that 
\begin{equation}\label{approx_C_I}
||\phi_{i^*}(T_1-\delta)-u||_2^2 + ||\phi_{i^*-1}(T_1-\delta)-v||_2^2  \leq \eta/2\,.
\end{equation} 
This is a straightforward: Indeed we can approximate the $\phi_j(T_1-\delta)$ for $j\not\in \{i^*,i^*-1\}$
by the $\tilde{\phi}_j$ because of the first point of Proposition \ref{limit_phi_T_1}, thus we can choose two vectors $\{u,v\}$ in the two dimensional 
space $W$ so that \eqref{approx_C_I} holds. 
This completes the proof.
\end{proof}

We now turn to the study of the couple $(\tilde{\phi}_{i^*-1}(t),\tilde{\phi}_{i^*}(t))$ for $t\in [T_1-\delta;T_1)$ and in particular when $t\rightarrow T_1, t<T_1$. A crucial point is   equation \ref{div_integral} which we now prove.

It\^o's Formula gives for $t< T_1$
\begin{align*}
&\ln(\lambda_i-\lambda_{i-1})(t) = (- \gamma+ 2 p\beta) t + \int_0^{t} \sqrt{2}\frac{db_s^{i^*}-db_s^{i^*-1}}{(\lambda_{i^*}-\lambda_{i^*-1})(s)} \\
&- p\beta \int_0^{t} \sum_{j\neq i^*,i^*-1} 
\frac{ds}{(\lambda_{i^*}-\lambda_j)(\lambda_{i^*-1}-\lambda_j)(s)} - \int_0^{t} \frac{2\,ds}{(\lambda_{i^*}-\lambda_{i^*-1})^2(s)}\,.
\end{align*}
If we suppose that $\int_0^{T_1} dt/(\lambda_{i^*}-\lambda_{i^*-1})^2(t) < + \infty$ and since $T_1<\tau^3_\epsilon$ for some $\epsilon>0$ small enough, we obtain a contradiction letting $t\rightarrow T_1$: under this assumption, the right hand side 
tends to 
$-\infty$ whereas the left hand side is almost surely bounded in this limit.  
\qed

The next Lemma \ref{study_tilde_phi} shows that the orthonormal basis $(\tilde{\phi}_{i^*-1}(t),\tilde{\phi}_{i^*}(t))$ of the subspace $W$ is in fact uniformly distributed 
in the set of all orthonormal basis of $W$ in the limit $t\rightarrow T_1,t<T_1$.

As $W$ is two dimensional, up to a change basis, we can suppose that the two vectors $\tilde{\phi}_{i^*-1}(t)$ and $\tilde{\phi}_{i^*}(t)$ are 
two dimensional (we just study the evolution of their coordinates in an orthonormal basis of $W$). Let us define the two by two matrix $\tilde{\phi}(t)$ 
whose first line is the vector $\tilde{\phi}_{i^*}(t)$ and second line is the vector $\tilde{\phi}_{i^*-1}(t)$:
\begin{equation*}
\tilde{\phi}(t) := \begin{pmatrix}
\tilde{\phi}_{i^*}(t) \\
\tilde{\phi}_{i^*-1}(t)
\end{pmatrix}\,.
\end{equation*}

\begin{lemma}\label{study_tilde_phi}
The matrix $\tilde{\phi}(t)$ converges in law when $t\rightarrow T_1,t<T_1$ to the Haar probability measure on the orthogonal group (respectively unitary group if $\beta=2$.)
\end{lemma}
\begin{proof} 

To simplify notations, we do the proof in the case $\beta=1$.

Set $t_0:= T_1-\delta$ and define for $t\in[0;\delta)$ the function 
\begin{equation*}
\varphi(t):= \int_{t_0}^{t_0+t} \frac{ds}{(\lambda_{i^*}-\lambda_{i^*-1})^2(s)}
\end{equation*}
and denote by $\varphi^{-1}$ its functional inverse. We now proceed to a change of time by setting for $t\in[0;\delta)$
\begin{align*}
\tilde{\psi}_{i^*}(t) = \tilde{\phi}_{i^*}(\varphi^{-1}(t)) , \quad \tilde{\psi}_{i^*-1}(t) = \tilde{\phi}_{i^*-1}(\varphi^{-1}(t))\,.
\end{align*}
As $\varphi^{-1}(t) \rightarrow + \infty$ when $t\rightarrow \delta, t<\delta$ (because of \eqref{div_integral}), 
the two by two matrix $\tilde{\psi}(t)$ whose first line is $\tilde{\psi}_{i^*}(t)$ and second line is $\tilde{\psi}_{i^*}(t)$:
\begin{equation*}
\tilde{\psi}(t) := \begin{pmatrix}
\tilde{\psi}_{i^*}(t) \\
\tilde{\psi}_{i^*-1}(t)
\end{pmatrix}
\end{equation*}
is now defined for all $t\in \R_+$ and verifies the following stochastic differential equation
\begin{equation}\label{tilde_psi}
d\tilde{\psi}(t) = \sqrt{p}\, A \, \tilde{\psi}(t)\, dB_t - \frac{p\beta}{2} \,\tilde{\psi}(t) \, dt\,.
\end{equation}
where $B$ is a standard Brownian motion on $\R$ and 
where $A$ is the two by two matrix defined by
\begin{equation*}
A=\begin{pmatrix}
0 & 1 \\
-1 & 0
\end{pmatrix}\,.
\end{equation*}
Note in particular that $A^2=-I$. 

It is clear that there is pathwise uniqueness in the stochastic differential equation \eqref{tilde_psi} (it is linear in 
$\tilde{\psi}$). Therefore to solve entirely this equation, we just need to exhibit one solution. Using It\^o's Formula, one can check that the solution is 
\begin{align*}
\tilde{\psi}(t) &=  \exp\left(\sqrt{p}\, A \, B_t \right)\tilde{\psi}(0) \\
&=  \begin{pmatrix}
\cos(\sqrt{p}B_t) & \sin(\sqrt{p}B_t) \\
-\sin(\sqrt{p}B_t) & \cos(\sqrt{p}B_t)
\end{pmatrix}\tilde{\psi}(0)\,.
\end{align*}
Note that for all $t\in \R_+$, the matrix $\tilde{\psi}(t)$ is indeed in the space of orthogonal matrices. 

But $(\cos(\sqrt{p}B_t), \sin(\sqrt{p}B_t))$ converges in law as time goes to infinity towards the 
law of $(\theta,\varepsilon \sqrt{1-\theta^2})$ with $\theta$ uniformly distributed on $[-1,1]$ and $\varepsilon=\pm 1$ with probability $1/2$, from which the result follows.
\end{proof}

Lemmas \ref{comp_phi_phi_tilde} and \ref{study_tilde_phi} give the second statement of Proposition \ref{limit_phi_T_1}.

\section*{Acknowledgments}
We thank J.-P. Bouchaud for proposing this subject, and working with us to understand it,  which  led to \cite{PRL}; the present article gives complete proofs of some of the results  stated in \cite{PRL}. We are very grateful to C. Garban who proposed us to attack the analysis of the case $p\beta<1$ by introducing the auxiliary process $\lambda^\delta$ of section \ref{auxiliaryprocesssec} and for many enlightening discussions. We thank  L. Dumaz, R. Rhodes, V. Vargas, G . Schehr, R. Chicheportiche and F. Benaych-Georges for useful comments and discussions.


\begin{thebibliography}{20}
\bibitem{PRL}R. Allez, JP Bouchaud and A Guionnet,
Invariant $\beta$-ensembles and the Gauss-Wigner crossover, {\it preprint}
\bibitem{AGZ} G.W. Anderson, A. Guionnet and O. Zeitouni, {\it An Introduction to Random Matrices}, Cambridge Studies in Advanced Mathematics, 
Cambridge University Press (2009). 
\bibitem{Billingsley} P. Billingsley, Convergence of Probability Measures (Wiley Series in Probability and Statistics). 
\bibitem{Cepa} E. C\'epa and D. L\'epingle, Diffusing particles with electrostatic repulsion, Probability Theory and Related Fields, 1997.  
\bibitem{Cepa2} E. C\'epa and D. L\'epingle, No multiple collisions for mutually repelling Brownian particles, S\'eminaire de Probabilit\'es 40 (2007) 241-246.
\bibitem{Dumitriu} I. Dumitriu, A. Edelman, {\it Matrix Models for Beta Ensembles}, Journal of Mathematical Physics 43, 5830--5847 , 2002.
\bibitem{Jacod} J. Jacod and A.N. Shiryaev, Limit theorems for stochastic processes. Springer, Berlin Heidelberg New York 1987. 
\bibitem{karatzas} I. Karatzas and S. E. Shreve, Brownian Motion and Stochastic Calculus, Second Edition, Graduate Texts in Mathematics, Springer.
\bibitem{Mehta} M. L. Mehta, Random matrices,Academic press, 1991.
\bibitem{Yor} D. Revuz and M. Yor, Continuous Martingales and Brownian Motion, Third edition, Springer. 
\end{thebibliography}
\end{document}